\pdfoutput=1
\documentclass[11pt]{amsart}
\usepackage{amsmath,amssymb,mathtools}
\usepackage[margin=1.05in]{geometry}
\usepackage[hypertexnames=false]{hyperref}
\usepackage{booktabs}
\usepackage{longtable}
\usepackage{array}
\usepackage{siunitx}
\usepackage{algorithm,algpseudocode}
\usepackage{graphicx}
\graphicspath{{figures/}}

\newtheorem{theorem}{Theorem}[section]
\newtheorem{lemma}[theorem]{Lemma}
\newtheorem{proposition}[theorem]{Proposition}
\newtheorem{corollary}[theorem]{Corollary}
\theoremstyle{definition}
\newtheorem{definition}[theorem]{Definition}
\newtheorem{remark}[theorem]{Remark}

\newcommand{\RR}{\mathbb{R}}
\newcommand{\QQ}{\mathbb{Q}}
\newcommand{\CC}{\mathbb{C}}
\newcommand{\ZZ}{\mathbb{Z}}
\newcommand{\NN}{\mathbb{N}}
\newcommand{\cO}{\mathcal{O}}
\newcommand{\rg}{\operatorname{rg}}
\newcommand{\Id}{\mathrm{I}}
\newcommand{\brem}{\mathbin{\%}}
\newcommand{\RN}[1]{\textup{\uppercase\expandafter{\romannumeral#1}}}
\newcommand{\nR}{R^n}
\newcommand{\nm}[1]{{#1}^n_{m}}
\newcommand{\npmz}[2]{{#1}^n_{(#2)}}
\newcommand{\cR}{\mathcal{R}}
\newcommand{\Sone}{\npmz{S}{1}}
\DeclareMathOperator{\blockdiag}{blockdiag}

\title[Explicit oblique inverses of the ADRT]{Explicit oblique inverses of the approximate discrete Radon transform}
\author{Donsub Rim}
\address{Department of Mathematics, Washington University in St.~Louis, St.~Louis, MO 63130, USA}
\email{rim@wustl.edu}
\date{\today}

\begin{document}
\begin{abstract}
The approximate discrete Radon transform (ADRT) is a fast, hierarchical
approximation of the Radon transform whose levels have explicit spectral
decompositions, but its Moore--Penrose inverse, needed for data outside the
range, has no closed form.  We construct explicit left inverses of the ADRT,
oblique projections onto its range, that are exact for every choice of their
parameters and amplify noise by a constant factor relative to the
Moore--Penrose inverse.  Realized with polynomial filters on the butterfly
structure of the levels, they are fixed linear maps of cost
$\cO(N^{5/2}\log^{3/2}N/\sigma_{\min})$ for an $N\times N$ image, where
$\sigma_{\min}$ is the smallest singular value of the ADRT; with the
measured growth $\sigma_{\min}\sim N^{1/4}$ this is $\cO(N^{2.25}\log^{3/2}N)$.
A hybrid that
keeps the explicit inverse of the cross-quadrant level is cheaper and closer to
the Moore--Penrose inverse.  In compiled implementations both outperform a
polynomial solver of the normal equations from $N=256$ on, and their
conditioning is traced to the aliasing of integer-sampled lines.
\end{abstract}
\maketitle

\section{Introduction}

The ADRT \cite{Gotz96,Brady98} approximates the continuous Radon transform
\cite{Natterer} by sums of pixel values over \emph{digital lines}, and it is
fast: $\cO(N^2\log N)$ operations for an $N\times N$ image, $N=2^n$ (here and
throughout $\log$ is the logarithm of base two).  The single-quadrant ADRT has
an exact and fast inverse \cite{Rim20}, its range over square images is
characterized by $N(N-1)/2$ linear constraints \cite{LRR}, and each level of
the full ADRT has an explicit spectral decomposition which yields the explicit
inverse SPIFE of complexity $\cO(N^2\log^2N)$ for data in the range
\cite{LORR}.  For general data the problem of interest is
\begin{equation}\label{eq:prob}
\min_{x\in\RR^{N\times N}}\lVert\nR[x]-b\rVert ,\qquad
b\in\RR^{4\times N\times(\frac32N-\frac12)},
\end{equation}
whose solution $(\nR)^\dagger b$ is computed iteratively (CG, or the multigrid
of Press \cite{Press06}; see \cite{LORR,Otness23}), and \cite{LORR} lists as
open the construction of an explicit fast projection to the range.

\subsubsection*{Relation to earlier work.}  This paper continues two earlier
studies of the ADRT.  The first \cite{LRR} characterizes the range of the
single-quadrant ADRT over square images: a datum lies in the range exactly
when it satisfies $N(N-1)/2$ explicit linear constraints, mass conditions
between neighbouring slopes and support conditions on the offsets, derived
level by level from the arithmetic of digital lines; out of range, the exact
inverse of \cite{Rim20} produces images that grow polynomially in the offset.
Proposition~\ref{prop:pole} below is the Fourier form of that statement, a
pole of order $\log N$ at offset frequency zero, and the polynomial filters
of Section~\ref{sec:poly} are a way of handling it without the growth.  The
second \cite{LORR} factors the ADRT into levels that act as convolutions up to
permutations, gives each single-quadrant level an explicit singular value
decomposition by discrete sine transforms, derives a spectral decomposition
of the cross-quadrant level from a block-sinusoidal basis, and assembles the
explicit inverse SPIFE of cost $\cO(N^2\log^2N)$, competitive with iterative
methods for data in the range; it also observes that SPIFE is unstable for
data outside the range and leaves open an explicit projection onto the
range.  The offset-Fourier butterflies of Section~\ref{sec:fourier} are the
level decompositions of \cite{LORR} in a cyclic setting, and the hybrid of
Section~\ref{sec:hyb} keeps its cross-quadrant inverse.

Our approach is spectral, and its continuous counterparts are classical.  The
Radon transform and its inversion are treated in
\cite{Radon1917,Cormack63,Deans83,Helgason99,Natterer,Epstein07}; singular value
decompositions are known for full, exterior, interior and fan-beam geometries
\cite{Marr74,Quinto83,Louis84,Maass87,Maass92,KazantsevBukhgeim04}, and
regularization by spectral filtering, in particular the Tikhonov--Phillips
regularization of the Radon transform \cite{Louis85} and the design of filters
\cite{LouisSchuster96}, is the continuous analogue of the filters $p_k$ and of
the family $L^n_\tau$ below.  The single quadrant sees a limited range of
angles, whose instability is analysed in \cite{LouisRieder89}.  The cyclic
offsets of Section~\ref{sec:fourier} turn each quadrant into a Radon
transform on a torus; for such periodic transforms Fourier analysis is
explicit \cite{Ilmavirta15,Ilmavirta20,Railo20,IlmavirtaMonard19}, and discrete
periodic Radon transforms have exact fast inverses
\cite{Kelley93,Hsung96,Matus93}.  Fast but ill-conditioned formal inverses
also exist for Fourier-based Radon transforms in two and three dimensions
\cite{OlivaGarcia22,Averbuch03}, which is the same tension between speed and
stability that motivates this paper.  The ADRT serves as a fast Radon
transform in dimensional splitting of hyperbolic equations
\cite{Rim18split}, displacement interpolation \cite{Rim18mr} and sliced
Wasserstein barycenters \cite{Bonneel15}, and its continuous model underlies
electron tomography \cite{Frank96,Midgley03}; in these uses data are rarely
exactly in the range, which is where an inverse that is stable for general
data is needed.  For the Chebyshev semi-iteration used throughout see
\cite{GolubVarga61,Greenbaum97,Saad}.

\subsubsection*{Other discretizations.}  It is natural to ask which
discretizations of the Radon transform come with an inverse that is explicit,
fast and stable for data outside the range, and which have an explicit
description of their range.  Transforms built on the geometry of a finite
affine plane have both: for the finite Radon transform on $\ZZ_p\times\ZZ_p$
\cite{Matus93}, which sums over the $p(p+1)$ lines $\{(x,mx+b)\}$ and $\{(c,y)\}$
modulo $p$, any two distinct points lie on exactly one line and each point on
$p+1$ lines, so $R^\top R=p\,\Id+\mathbf 1\mathbf 1^\top$ ($\mathbf 1$ the vector of ones),
$R^\dagger=\frac1p\big(\Id-\frac{\mathbf 1\mathbf 1^\top}{p(p+1)}\big)R^\top$ is explicit and costs one
back-projection, and the range is cut out by the $p$ conditions that all $p+1$
projections have the same total sum.  The same algebra underlies
the variants of Kingston and Svalbe \cite{Kingston06}.  The Mojette transform
has exact linear-time inverses for Katz-admissible angle sets
\cite{Guedon05}, but these are sensitive to noise, and its robust inverses are
either much slower or iterative.  In all of these the lines wrap around the
periodic array, so the transforms are not approximations of the continuous
Radon transform in the sense of \cite{Brady98}.

Among discretizations that keep the Euclidean geometry of the lines, Beylkin's
discrete Radon transform \cite{Beylkin87} sums over a family of slopes with
periodic offsets; its matrix is block circulant, and its least-squares inverse
$(R^*R)^{-1}R^*$ is computed exactly, one offset frequency at a time, on the
band where the per-frequency systems are invertible.  This is the structure of
the single-quadrant transforms in Lemma~\ref{lem:blockdiag} and of the
quadrant-wise pseudo-inverses $\cR^+_q$ below; it does not survive the
combination of quadrants, which is where the present construction begins.  The
pseudo-polar Fourier transform and the fast slant stack \cite{Averbuch08}
follow the lines through the Fourier slice theorem.  They are inverted
iteratively by preconditioned conjugate gradients, with tens of iterations in
the reported experiments, and also \emph{directly}, by a fixed sequence of
well-conditioned one-dimensional least-squares resamplings from the
pseudo-polar onto a Cartesian grid, in $\cO(N^2\log N)$ operations
\cite{Averbuch08,Averbuch16}; applied to arbitrary data, the direct algorithm
is a fast, stable left inverse with a fixed operation count, of the same kind
as the inverses constructed here, though its noise amplification relative to
the least-squares solution has not, to our knowledge, been analysed.  For polar
Fourier sampling, which is too ill-conditioned for direct inversion, direct
least-squares solvers based on hierarchical semiseparable matrices have
recently been proposed \cite{LiLiu26}.  None of these Fourier-based
transforms, to our knowledge, comes with an explicit description of its
range: data on the pseudo-polar grid are in the range exactly when they are
the samples of a trigonometric polynomial of the right degree, a condition
that the resampling checks only implicitly, and for the continuous transform
the range is described by the Helgason--Ludwig moment conditions
\cite{Natterer,Helgason99}, of which sampled data satisfy discretized versions
only approximately.  (For the planar X-ray transform the range has also been
characterized through the Fourier coefficients of the data on the torus
\cite{SadiqTamasan22}.)

The ADRT, whose lines are integer-sampled and do not wrap, differs from both
groups.  Its range is described explicitly: for a single quadrant by the
$N(N-1)/2$ mass and support conditions of \cite{LRR}, linear constraints with
integer coefficients that are local in the offset and follow the levels of
the transform.  Its inverses, on the other hand, were either exact only on the
range \cite{Rim20,LORR} or iterative (Press's multigrid \cite{Press06},
conjugate gradients \cite{Otness23}), and the fast formal inverses are severely
ill-conditioned (Press estimates a condition number of about $10^{18}$ at
$N=256$, and see \cite{OlivaGarcia22}).  The construction of this paper gives an
explicit inverse of the ADRT that is stable for general data, in closed form
and with a proved complexity bound, a small power of $N$ above the optimal
$N^2\log^2N$ (Section~\ref{sec:cond}), by changing the question from the
orthogonal to an oblique projection onto the range.

This paper takes the oblique point of view.  A left inverse $L$ of $\nR$
($L\nR=\Id$) solves \eqref{eq:prob} exactly when $b\in\rg(\nR)$ and, for
$b=\nR[x]+\nu$, returns $x+L\nu$; $\nR L$ is the oblique
projector onto $\rg(\nR)$ along $\ker L$.  SPIFE is such an $L$, with
$\ker L$ spanned by the forward lifts of the level kernels $\ker{\npmz{S}{m}}^\top$.
Its instability comes from the small singular values $2\sin\frac{k\pi}{4N+2}$ of
the levels, which correspond to low offset frequencies in ``anti-phase''
combinations of adjacent strips \cite{LORR}.  Our construction replaces the
level-wise inversion by a quadrant-wise one in the offset-Fourier domain, keeps
each quadrant only on the subspace where it is stable (a spectral projector of
the quadrant's Gram matrix at every offset frequency), and recombines the four
quadrants through a well conditioned frame operator (the sum of the four
projectors, Section~\ref{sec:partition}).

The resulting family of left inverses $L^n_\theta$ (Theorem~\ref{thm:main}) has
the following properties.
\begin{enumerate}
\item[(i)] $L^n_\theta\nR=\Id$ exactly, for every threshold $\theta$; the proof
rests on one identity, $P_kA_k^+A_k=P_k$, and the same argument admits any
filter of the quadrant Gram matrices in place of the spectral projector
(Remark~\ref{rem:filter}), giving a one-parameter family $L^n_\tau$ that
tends to the Moore--Penrose inverse as $\tau\to\infty$.
\item[(ii)] Every factor of $L^n_\theta$ is given in closed form from the
offset-Fourier (butterfly) representation of the single-quadrant levels
(Lemma~\ref{lem:blockdiag}), except one symmetric positive definite operator,
the window-compressed frame operator, whose condition number stays below $10$
for $N\le256$ and whose inverse is therefore realized by a Chebyshev polynomial
of fixed, data-independent degree.
\item[(iii)] The amplification of data noise is a constant multiple of that
of the Moore--Penrose inverse: $1.8$ for the spectral projectors, $1.36$ and
$1.07$ for the Tikhonov filter with $\tau=0.2N$ and $\tau=N$, independently of
$N$ up to $256$ (Sections~\ref{sec:stab} and \ref{sec:num}).
\item[(iv)] With the filters taken as Chebyshev polynomials of the quadrant
Gram matrices, applied by the same butterflies that compute the ADRT, the
degrees needed at the offset frequencies sum to $\cO(N\sqrt{\log N})$, which
is proved for the explicit bounds \eqref{eq:Lambda}, and the total cost is
$\cO(N^2\log^{3/2}N\cdot\sqrt{\kappa(E^\top F_pE)})$ for fixed precision
(Theorem~\ref{thm:fast}), where $\kappa(E^\top F_pE)\approx\tau/\lambda_{\min}(R^{n\top}R^n)$
is governed by the smallest singular value of the ADRT (Section~\ref{sec:cond}); no iteration is run to a tolerance, all polynomial
degrees being fixed in a data-independent setup for which explicit formulas
are given (Section~\ref{sec:poly}).
\item[(v)] Keeping the explicit inverse of the cross-quadrant level $\Sone$ of
\cite{LORR} and filtering only the single-quadrant levels gives a second exact
left inverse, the hybrid of Section~\ref{sec:hyb}, which costs $24$--$30\%$ less
per application and amplifies noise by $1.20$ instead of $1.36$ at $\tau=0.2N$.
On the way the level-1 Gram matrix ${\Sone}^\top\Sone$ is shown to be the
nearest-neighbour stencil $8\Id+K\otimes J+J\otimes K$ with spectrum in $[4,16]$
(Lemma~\ref{lem:g1}), and the projection trick of \cite[Thm.~5.1]{LORR} to be
the inverse of $4\Id+8\Pi$, $\Pi$ the $2\times2$-cell average (Lemma~\ref{lem:lorr}).
\end{enumerate}
\subsubsection*{Notes on terminology.}
The inverses are \emph{explicit} in the following
sense: each is a fixed linear map, given in closed form with every factor
explicit, in which the two matrix functions ($p_k(M_k)$ per frequency and
$(E^\top F_pE)^{-1}$) are replaced by Chebyshev polynomials of predetermined
degree, evaluated by their three-term recurrences.  These recurrences are
semi-iterations in the sense of Golub and Varga \cite{GolubVarga61}, and we do
not call the construction non-iterative; what distinguishes it from an
iteration stopped at a tolerance is that the left-inverse identity (i) holds
exactly for every degree, that the degrees are fixed a priori and
independently of the data, and that no residual is ever tested.
\subsubsection*{Press's conjecture.}  Press observed that his full multigrid
inverse reduces the error by a fixed factor per $\cO(N^2\log^2N)$ operations,
reaching a given accuracy in $\cO(N^2\log^3N)$ work, and conjectured that an
iterative inverse of complexity $\cO(N^2\log^qN)$ with $2\le q\le3$ exists, $q=2$
being optimal, i.e.\ $\cO(N^2\log^2N\,\log\varepsilon^{-1})$ for accuracy $\varepsilon$
\cite{Press06}; see the accounts in \cite[\S1]{LRR} and \cite[\S1]{LORR}, where
SPIFE attains $q=2$ for data in the range.  These bounds were not proved.
The inverse constructed here costs
$\cO(N^{5/2}\log^{3/2}N\log\varepsilon^{-1}/\sigma_{\min}(R^n))$ by a fixed-degree polynomial
formula: the factor $N^2\log^{3/2}N\log\varepsilon^{-1}$ for the frame operator is
proved (Theorem~\ref{thm:fast}(iii) with the explicit bounds \eqref{eq:Lambda},
which fix the degrees), and the remaining factor, the degree of the outer
polynomial, is $\sqrt{\kappa(E^\top F_pE)}\log\varepsilon^{-1}$ with
$\kappa(E^\top F_pE)\approx\tau/\lambda_{\min}(R^{n\top}R^n)$ (Section~\ref{sec:cond}).  The measured $\sigma_{\min}(R^n)\sim N^{1/4}$ makes the cost about
$\cO(N^{2.25}\log^{3/2}N)$, so that the conjectured $q=2$ is missed by a small
power of $N$, not by a logarithm.  Section~\ref{sec:cond} traces this factor to the digital
lines themselves, namely to the aliasing of integer-sampled lines at offset
frequencies near $2\pi/3$, rather than to the inverse: the construction has
cost $\cO(N^2\log^{3/2}N\log\varepsilon^{-1})$ for any transform of the same
structure whose normal operator has $\lambda_{\min}\gtrsim N$.  Two further qualifications: the inner
filters need only a fixed, crude tolerance ($\varepsilon_{\rm in}=0.1$ changes neither
the left-inverse identity nor the noise amplification,
Section~\ref{sec:inner}), so that $\varepsilon$ enters through the outer degree
alone; and the inverse is oblique, with a constant factor
$1.06$--$1.4$ in noise amplification over the Moore--Penrose inverse.  For the
Moore--Penrose inverse itself the methods of \cite{Press06,LORR} and the
symbol-preconditioned iterations remain iterative, with
observed complexities $\cO(N^2\log^3N)$ to $\cO(N^2\log^4N)$ in the same range of
$N$, where $\sigma_{\min}(R^n)^{-1}$ and a power of $\log N$ are not
distinguishable.  A fixed-degree polynomial alternative for the Moore--Penrose inverse,
a fixed-degree Chebyshev polynomial of the normal equations, has degree
$\Theta(N/\sigma_{\min}(R^n))$ and cost $\cO(N^3\log N/\sigma_{\min}(R^n))$,
asymptotically more than the oblique inverse; with compiled
implementations on one core it is faster up to $N=128$ and slower from
$N=256$ on (Remark~\ref{rem:normal}, Section~\ref{sec:compiled}).  

What the present construction shows is that the obstruction
to an explicit inverse is confined to the inversion of one frame operator
whose conditioning is that of the ADRT itself.
The instability that these constructions avoid is the Fourier dual of the
polynomial growth of out-of-range data in \cite{LRR}: the exact single-quadrant
inverse has a pole of order $\log N$ at offset frequency zero
(Proposition~\ref{prop:pole}).  Proved are property (i), the closed forms in (ii)
and the degree bound in (iv); the quantitative statements about conditioning,
noise amplification and cost, including (iii), are backed by numerical tests up to
$N=256$, including a Shepp--Logan example with up to $10\%$ noise.

The paper is organized as follows.  Section~\ref{sec:prelim} fixes the
notation of \cite{LRR,LORR}; Section~\ref{sec:fourier} diagonalizes the cyclic
levels in the offset frequency and identifies the pole at $\omega=0$;
Section~\ref{sec:oblique} constructs the oblique inverses and
Section~\ref{sec:poly} their polynomial-filter realization, including
explicit and certified spectral bounds; Section~\ref{sec:hyb} is the hybrid;
Section~\ref{sec:num} collects the numerical results of both constructions;
Section~\ref{sec:cond} analyses what governs their conditioning, and
Section~\ref{sec:disc} concludes.
\section{Notation and preliminaries}\label{sec:prelim}

We follow the notations from \cite{LRR,LORR}.
Let $[m]:=\{i\in\NN_0: i<m\}$, $[m]_1:=\{i+1:i\in[m]\}$,
$a\brem b$ the remainder of $a$ modulo $b$, and $\Id_m$ the identity in
$\RR^{m\times m}$.  Images are $x\in\RR^{N\times N}$ with entries $x(i,j)$,
$(i,j)\in[N]\times[N]$, $i$ the row index.

\subsection{Digital lines and the single-quadrant ADRT}
The digital lines $D_n(h,s)$ determined by slope $s\in[N]$ and offset $h\in\ZZ$ (also called the height)
are the subsets of $\ZZ\times[N]$ defined recursively by
\cite{Brady98,Rim20,LORR}
\begin{equation}\label{eq:dline}
D_n(h,2s)=D^{(L)}_{n-1}(h,s)\cup D^{(R)}_{n-1}(h+s,s),\qquad
D_n(h,2s+1)=D^{(L)}_{n-1}(h,s)\cup D^{(R)}_{n-1}(h+s+1,s),
\end{equation}
where $D^{(L)}_{n-1},D^{(R)}_{n-1}$ are digital lines of half length on the left
and right halves of the columns, initialized by $D_0(h,0)=\{(h,0)\}$.  The
single-quadrant ADRT of $x$ is
\begin{equation}\label{eq:Rnn}
R^n_n[x](h,s):=\sum_{(i,j)\in D_n(h,s)}x(i,j),\qquad (h,s)\in E^n,
\end{equation}
where, with $E^n_m:=\{(h,s)\in\ZZ\times[2^n]: -(s\brem2^m)\le h<2^n\}$ as in
\cite{LORR}, $E^n:=E^n_n$ is the set of all lines meeting the square,
$|E^n|=\frac32N^2-\frac12N$.  The recursion \eqref{eq:dline} gives
$R^n_n=S^n_n\cdots S^n_1$, where the level $\nm{S}:\RR^{E^n_{m-1}}\to\RR^{E^n_m}$
combines the sections $f^n_{m,2\ell},f^n_{m,2\ell+1}$ (strips of width $2^{m-1}$)
into $g^n_{m,\ell}$ (strips of width $2^m$) by \cite[eq.~(2.12)]{LORR}
\begin{equation}\label{eq:Snm}
    g^n_{m,\ell}(h,2s)=f^n_{m,2\ell}(h,s)+f^n_{m,2\ell+1}(h+s,s),
    \quad
    g^n_{m,\ell}(h,2s+1)=f^n_{m,2\ell}(h,s)+f^n_{m,2\ell+1}(h+s+1,s).
\end{equation}

We shall use a closed form of \eqref{eq:dline}.  Write $s=\sum_{k\in[n]}s_k2^k$ and
$t=\sum_{j\in[n]}t_j2^j$ in binary.
\begin{lemma}[Bit-linear form]\label{lem:bitlinear}
$D_n(h,s)=\{(h+\delta_s(t),t):t\in[N]\}$ with
\begin{equation}\label{eq:delta}
\delta_s(t)=\sum_{j\in[n]}t_j\,w_j(s),\qquad
w_j(s)=\Big\lfloor\frac{s}{2^{\,n-j-1}}\Big\rfloor-\Big\lfloor\frac{s}{2^{\,n-j}}\Big\rfloor
=\bigg\lceil\frac12\Big\lfloor\frac{s}{2^{\,n-j-1}}\Big\rfloor\bigg\rceil .
\end{equation}
\end{lemma}
\begin{proof}
By \eqref{eq:dline} the top bit $t_{n-1}$ of the column contributes the jump
$\lfloor s/2\rfloor+(s\brem2)=\lceil s/2\rceil$, and the remaining bits see the
slope $\lfloor s/2\rfloor$ on a strip of half width; induction on $n$ gives
$\delta_s(t)=\sum_jt_j\lceil\lfloor s/2^{n-j-1}\rfloor/2\rceil$.  This is the
formula of \cite[Cor.~2.9]{LRR}, $k_+(t)-k_+(t-1)=((t)_2\wedge\neg(t-1)_2)*\overline{(s)}_2$,
summed over $t$.
\end{proof}
Thus $\delta_s(0)=0$, $\delta_s(N-1)=s$, and $|\delta_s(t)-st/N|\le n/2$: digital
lines deviate from straight lines of slope $s/N\in[0,1)$ by at most $\frac12\log N$.

\subsection{The full ADRT and its levels}
With the quadrant permutations of \cite[eq.~(2.19)]{LORR},
\begin{equation}\label{eq:permute}
\begin{aligned}
(T^n_{\RN{1}}x)(i,j)&=x(j,2^n-1-i),&\qquad (T^n_{\RN{2}}x)(i,j)&=x(2^n-1-i,j),\\
(T^n_{\RN{3}}x)(i,j)&=x(i,j),&\qquad (T^n_{\RN{4}}x)(i,j)&=x(2^n-1-j,2^n-1-i),
\end{aligned}
\end{equation}
the full ADRT is $\nR:=\npmz{R}{n}=\npmz{S}{n}\cdots\npmz{S}{1}$ with the
cross-quadrant level and the single-quadrant levels \cite[Def.~2.1]{LORR}
\[
\npmz{S}{1}=\big[S^n_1T^n_{\RN{1}},\,S^n_1T^n_{\RN{2}},\,S^n_1T^n_{\RN{3}},\,S^n_1T^n_{\RN{4}}\big]^\top,\qquad
\npmz{S}{m}=\Id_4\otimes\nm{S}\quad(m\ge2);
\]
equivalently $\nR[x]=\big(R^n_n[T^n_qx]\big)_{q=\RN{1},\dots,\RN{4}}\in\RR^{4\times E^n}$.
(Storing each quadrant as an array of shape $(2N-1)\times N$ indexed by
$h+N-1$ and $s$, with zeros at $h<-s$, gives the layout of \cite{Otness23};
the structurally zero entries affect no generalized inverse.)  $\nR$ has full
column rank since $R^n_n$ is injective \cite{Rim20}, and
$\kappa(\nR)=1.53,2.46,3.75,5.96,9.89$ for $N=2,\dots,32$, i.e.\ mildly
ill-conditioned, whereas $\kappa(R^n_n)$ grows faster than any power of $N$
(verified by dense SVD).

\cite[Thm.~3.5]{LORR} gives the SVD $\npmz{S}{m}=\npmz{U}{m}\npmz{\Sigma}{m}{\npmz{V}{m}}^\top$
for $m\ge2$ from the SVD of the bidiagonal convolution $K^+_{2N}$ (singular
values $2\cos\frac{k\pi}{4N+2}$, $k\in[2N]_1$, DST-II/DST-I singular vectors), and
\cite[Thm.~4.1, 5.1]{LORR} a spectral decomposition of $\npmz{S}{1}$; the
resulting SPIFE,
\begin{equation}\label{eq:spife}
(\nR)^\dagger_{\rm SPIFE}=\widehat V^n_{(1)}\widehat U^{n\top}_{(1)}\prod_{m=2}^n\npmz{V}{m}(\npmz{\Sigma}{m})^{-1}{\npmz{U}{m}}^\top ,
\end{equation}
is a left inverse of $\nR$ \cite[Thm.~5.3]{LORR}.  Clearly, it is not the Moore--Penrose
inverse: for $b\notin\rg(\nR)$ one has $(\nR)^\dagger=L\,P_{\rg(\nR)}$ for every left
inverse $L$, and the orthogonal projector $P_{\rg(\nR)}=\nR({\nR}^\top\nR)^{-1}{\nR}^\top$ has
no closed form: the eigenvalues of ${\nR}^\top\nR$ generate non-abelian Galois
extensions already for $N=4$, so that no trigonometric transform diagonalizes
${\nR}^\top\nR$ (Appendix~\ref{app:galois}).
\section{The levels in the offset-Fourier domain}\label{sec:fourier}

\subsection{Cyclic offsets}
Let $L=2N$ (any $L \ge 2N$ will do) and consider images on the cyclic strip
$\ZZ_L\times[N]$, i.e.\ $f\in\CC^{L\times N}$ with the offset $h$ taken modulo $L$.
Define the cyclic single-quadrant ADRT by \eqref{eq:Rnn} with
$(h+\delta_s(t))\brem L$ in place of $h+\delta_s(t)$,
\begin{equation}\label{eq:cyc}
\cR^n_n[f](h,s)=\sum_{t\in[N]}f\big((h+\delta_s(t))\brem L,\,t\big),\qquad (h,s)\in[L]\times[N].
\end{equation}
If $f$ vanishes outside the rows $[N]$ then no line wraps (every line meeting the
square has $-N<h<N$ and rises by at most $N-1$), so $\cR^n_n[f]$ at
$h\in\{-(N-1),\dots,N-1\}\brem L$ equals $R^n_n$ of the square image and vanishes
elsewhere.  Let $\iota:\RR^{E^n}\to\RR^{L\times N}$ be this embedding of data
($h\mapsto h\brem L$, zeros elsewhere) and $E:\RR^{N\times N}\to\RR^{L\times L}$ the
zero padding of an image into the $L\times L$ torus.

\subsection{Diagonalization in the offset and the butterfly}
For $f\in\CC^{L\times N}$ let
\[
\hat f(k,t):=\sum_{h\in[L]}f(h,t)\,e^{-2\pi\mathrm{i}kh/L},\qquad k\in[L],
\]
be the discrete Fourier transform in the offset ($\mathrm{DFT}_h$), and
$\omega_k:=2\pi k/L$.

\begin{lemma}[Block diagonalization of the cyclic levels]\label{lem:blockdiag}
\mbox{}
\begin{enumerate}
\item[(i)] For every $k\in[L]$,
\begin{equation}\label{eq:Ak}
\widehat{\cR^n_n[f]}(k,s)=\sum_{t\in[N]}e^{\mathrm{i}\omega_k\delta_s(t)}\hat f(k,t)
=\big(A_k\hat f(k,\cdot)\big)(s),\qquad A_k(s,t):=e^{\mathrm{i}\omega_k\delta_s(t)},
\end{equation}
i.e.\ $\cR^n_n=\mathrm{DFT}_h^{-1}\,\blockdiag_{k\in[L]}(A_k)\,\mathrm{DFT}_h$ and
$\cR^{n\top}_n=\mathrm{DFT}_h^{-1}\,\blockdiag_{k}(A_k^*)\,\mathrm{DFT}_h$, with $A_k\in\CC^{N\times N}$
indexed by (slope, column).
\item[(ii)] $A_k=\mathcal B^n_n\cdots\mathcal B^n_1$, where $\mathcal B^n_m$ is the level
$\nm{S}$ of \eqref{eq:Snm} at offset frequency $\omega_k$: it is block diagonal
over the strip pairs $\ell\in[2^{n-m}]$ and input slopes $s\in[2^{m-1}]$ with the
$2\times2$ blocks
\begin{equation}\label{eq:butterfly}
\begin{aligned}
\begin{pmatrix}\hat g(\cdot,2s)\\ \hat g(\cdot,2s+1)\end{pmatrix}
&=\begin{pmatrix}1& z^{s}\\ 1& z^{s+1}\end{pmatrix}
\begin{pmatrix}\hat f_{2\ell}(\cdot,s)\\ \hat f_{2\ell+1}(\cdot,s)\end{pmatrix}
=\beta\,\mathrm{diag}(1,z^{s})\begin{pmatrix}\hat f_{2\ell}(\cdot,s)\\ \hat f_{2\ell+1}(\cdot,s)\end{pmatrix},\\
z&=e^{\mathrm{i}\omega_k},\qquad \beta=\begin{pmatrix}1&1\\1&z\end{pmatrix}.
\end{aligned}
\end{equation}
\item[(iii)] $A_k$ and $A_k^*$ are applied to a vector in $\cO(N\log N)$ operations,
hence $\cR^n_n$ and $\cR^{n\top}_n$ in $\cO(N^2\log N)$.
\end{enumerate}
\end{lemma}
\begin{proof}
(i) By \eqref{eq:cyc}, $\cR^n_n$ 
commutes with the cyclic shift of $h$ (replacing
$f(h,t)$ by $f((h-1)\brem L,t)$ shifts $\cR^n_n[f]$ by one in $h$); for each
$(s,t)$ the map $f(\cdot,t)\mapsto f((\cdot+\delta_s(t))\brem L,t)$ is a cyclic
shift, which the DFT turns into multiplication by $e^{\mathrm{i}\omega_k\delta_s(t)}$.
Summing over $t$ gives \eqref{eq:Ak}; the statement for $\cR^{n\top}_n$ follows
by taking adjoints, since $\mathrm{DFT}_h/\sqrt L$ is unitary.
(ii) Apply (i) to the levels: by \eqref{eq:Snm}, $\nm{S}$ shifts the second
section by $s$ or $s+1$ in the offset and adds, which under $\mathrm{DFT}_h$ is
\eqref{eq:butterfly}; composing the $n$ levels gives $A_k$ because
$\cR^n_n=S^n_n\cdots S^n_1$ holds with cyclic offsets as well.  Equivalently, by
Lemma~\ref{lem:bitlinear}, $A_k(s,t)=\prod_{j\in[n]}\big(e^{\mathrm{i}\omega_kw_j(s)}\big)^{t_j}$
with $w_j(s)$ depending only on the top $j+1$ bits of $s$, so the sum over the
bits of $t$ factorizes level by level.
(iii) Each $\mathcal B^n_m$ has two nonzeros per row and there are $n=\log N$
levels.
\end{proof}

The singular values of $\beta$ are $2|\cos\frac{\omega_k}4|$ and
$2|\sin\frac{\omega_k}4|$: these are the two branches $2\cos\frac{k'\pi}{4N+2}$,
$k'\le N$ and $k'>N$, of the singular values of $K^+_{2N}$ in
\cite[Lemma~3.1]{LORR} (in-phase and anti-phase combinations of the two
sections at offset frequency $\omega_k$); Lemma~\ref{lem:blockdiag} is thus the
offset-Fourier form of the block-convolution structure
$\nm{S}=\nm{Q}\nm{Z}\nm{P}$ of \cite[Lemma~3.3]{LORR}.

\subsubsection*{Fourier cones.}  Transform also in the column index,
$\hat x(\omega,\eta)=\sum_{i,j}x(i,j)e^{-\mathrm{i}(\omega i+\eta j)}$, with
$\omega$ the row (offset) frequency and $\eta$ the column frequency.  For a
straight line of slope $\sigma\in[0,1]$, $i=h+\sigma j$, the sum along the
line of the plane wave $e^{\mathrm{i}(\omega i+\eta j)}$ equals
$e^{\mathrm{i}\omega h}\sum_je^{\mathrm{i}(\omega\sigma+\eta)j}$, which is of size $N$
when $\eta=-\sigma\omega$ and small otherwise: this is the Fourier slice
theorem.  As $\sigma$ ranges over the slopes of quadrant \RN{3}, the modes
seen by that quadrant fill the \emph{Fourier cone}
\[
\mathcal C_{\RN{3}}:=\{(\omega,\eta):\ |\eta|\le|\omega|,\ \omega\eta\le0\},
\]
and likewise $\mathcal C_{\RN{2}}=\{|\eta|\le|\omega|,\ \omega\eta\ge0\}$ for the
reflected quadrant and $\mathcal C_{\RN{1}},\mathcal C_{\RN{4}}$, obtained by
exchanging $\omega$ and $\eta$, for the transposed quadrants; the four cones
tile the frequency plane.  In the offset-Fourier picture of
Lemma~\ref{lem:blockdiag}, at a fixed $\omega_k$ the cone of quadrant \RN{3} is
the band $|\eta|\le|\omega_k|$ of column frequencies.  By a \emph{Fourier cone
weight} we mean a multiplier in $(\omega,\eta)$ supported on (a smoothed
version of) one cone; a family of such weights summing to one is a Fourier
partition of unity.

For real $\omega$ we write $A_\omega(s,t):=e^{\mathrm i\omega\delta_s(t)}$ and
$M_\omega:=A_\omega^*A_\omega$, so that $A_k=A_{\omega_k}$; two facts about the
Gram matrices $M_k:=A_k^*A_k=M_{\omega_k}$ are used below.
\begin{itemize}
\item $A_0$ is the all-ones matrix, $M_0$ has one eigenvalue $N^2$ (constants)
and $N-1$ zero eigenvalues.  For $\omega_k\ne0$, $M_k$ is positive definite:
the exact inversion formula \cite[Thm.~2.6]{LRR}, in Fourier form
\[
\begin{aligned}
\begin{pmatrix}\hat f_{2\ell}(\cdot,s)\\ \hat f_{2\ell+1}(\cdot,s)\end{pmatrix}
&=\frac1{z-1}\begin{pmatrix}z&-1\\ -z^{-s}&z^{-s}\end{pmatrix}
\begin{pmatrix}\hat g(\cdot,2s)\\ \hat g(\cdot,2s+1)\end{pmatrix}
=\mathrm{diag}(1,z^{-s})\,\beta^{-1}\begin{pmatrix}\hat g(\cdot,2s)\\ \hat g(\cdot,2s+1)\end{pmatrix},\\
\beta^{-1}&=\frac1{z-1}\begin{pmatrix}z&-1\\-1&1\end{pmatrix}
\end{aligned}
\]
(the cumulative sum is the division by $z-1$), inverts \eqref{eq:butterfly};
but the division by $z-1$ at each of the $n$ levels makes $\|A_k^{-1}\|$ as large as
$(2/\omega_k)^n$ in anti-phase directions.  This is the discrete limited-angle
problem \cite[Ch.~VI]{Natterer}: at frequency $\omega_k$ the single quadrant
sees its Fourier cone (the band of column frequencies $|\eta|\le|\omega_k|$) with
eigenvalues $\approx2\pi N/|\omega_k|$, and the complementary band with $\cO(1)$
eigenvalues; but, as the next paragraph and Section~\ref{sec:left} show, the
eigenvectors of $M_k$ are not exactly the plane waves of the cone.
\item Numerically, using power iteration for $N\le512$,
$\max_{\omega_k\ne0}\lambda_{\max}(M_k)|\omega_k|/N$ equals $5.8$, $6.1$, $7.0$, $8.1$, $9.2$, $10.6$ for
$N=16,\dots,512$, i.e.\ $\lambda_{\max}(M_k)\le C(N)N/|\omega_k|$ with
$C(N)\approx1.2\log N$ (for straight lines \eqref{eq:Ak} is the pseudo-polar
Fourier transform \cite{Averbuch08} and the constant is $2\pi$; the entries of
$M_k$ are given in closed form in Lemma~\ref{lem:kernel}).
\end{itemize}

\begin{remark}[The level decompositions do not compose]\label{rem:levels}
Lemma~\ref{lem:blockdiag}(ii) is the explicit spectral decomposition of the
levels of \cite{LORR} in the cyclic setting: every block of $\mathcal B^n_m$ is
$\beta\,\mathrm{diag}(1,z^s)$, so each level has the two singular values
$2|\cos\frac{\omega_k}4|$ and $2|\sin\frac{\omega_k}4|$, each of multiplicity $N/2$, with
explicit and locally supported singular vectors, namely the in-phase and
anti-phase combinations $(1,\pm e^{-\mathrm i\omega_k/2})/\sqrt2$ of a pair.  

One may ask
whether the spectrum of $M_k$ follows by composing these decompositions, since
the overlaps of the singular bases of consecutive levels are explicit.  They
are, and are exactly balanced: the outputs of $\mathcal B^n_{m-1}$ are paired
by slope within a strip while the inputs of $\mathcal B^n_m$ are paired by strip at
a fixed slope, so on the tensor factorization (strip)$\otimes$(slope) the two
levels act on complementary factors, and every singular vector of one level
has exactly half of its energy in the large-singular-value subspace of the
next ($|\langle u_\pm\otimes e_\ell,e_s\otimes v_\pm\rangle|^2=\frac14$).  The level
spectra therefore constrain only the extremes of $\mathrm{spec}(M_k)$:
$\|A_k\|\le(2\cos\frac{\omega_k}4)^n$ and $\sigma_{\min}(A_k)\ge(2\sin\frac{\omega_k}4)^n$,
the latter sharp (Proposition~\ref{prop:pole}).  The aligned level-product estimate
$\{(4\cos^2\frac{\omega}4)^{n-j}(4\sin^2\frac{\omega}4)^j\}$ reproduces the small end and roughly
the median of the spectrum, but not the top: at $N=64$ it gives
$\lambda_{\max}=4001,2804,855$ against the true $1590,398,216$ for
$\omega=0.25,1,2$, and a single spike followed by a gap where the true spectrum
has a plateau of $\approx4,12,30$ eigenvalues within a factor $10$ of $\lambda_{\max}$
(the Fourier cone, $\approx N\omega/\pi$).  It is this
plateau that fixes the stable subspaces and hence $\kappa(E^\top FE)$ in
Section~\ref{sec:oblique}, so the conditioning of the oblique inverses is not estimated by the
level decompositions alone; the correct level-by-level composition is the
operator recursion for the twisted Gram matrices in the proof of
Lemma~\ref{lem:kernel}.
\end{remark}

\subsection{The pole at $\omega=0$ and the polynomial growth of out-of-range data}
The instability of the single-quadrant inverse has a precise Fourier
description which connects the level singular values of \cite{LORR} to the
polynomial growth in the range characterization of \cite{LRR}.

\begin{proposition}\label{prop:pole}
For $0<\omega<2\pi$,
\[
\sigma_{\min}(A_\omega)\ \ge\ \big(2\sin\tfrac{\omega}{4}\big)^n .
\]
Hence $A_\omega^{-1}$ has a pole of order at most $n=\log N$ at $\omega=0$:
$\|A_\omega^{-1}\|\le(2\sin\frac\omega4)^{-n}\sim(2/\omega)^n$.
\end{proposition}
\begin{proof}
By Lemma~\ref{lem:blockdiag}(ii), $A_\omega=\mathcal B^n_n\cdots\mathcal B^n_1$ with
every $2\times2$ block of $\mathcal B^n_m$ unitarily equivalent to $\beta$, whose
singular values are $2\cos\frac\omega4$ and $2\sin\frac\omega4$; so
$\sigma_{\min}(\mathcal B^n_m)=2\sin\frac\omega4$ and submultiplicativity of the
smallest singular value gives the bound.
\end{proof}
Numerically the bound is sharp: $\sigma_{\min}(A_\omega)=(2\sin\frac\omega4)^n\,(1+\cO(\omega))$,
with ratios equal to $1.000$ to four digits for $\omega\le2^{-5}$ and $n\le5$, the
minimizing direction being $v(t)=(-1)^{|t|}e^{-\mathrm{i}\omega t/2}(1+\cO(\omega))$,
$|t|$ the number of ones in the binary expansion of $t$.  The pole therefore
has order exactly $n$: $\|A_\omega^{-1}\|\sim(2/\omega)^n$.

The $n$ factors $(2\sin\frac\omega4)^{-1}$ are the $n$ divisions by $z-1$ in
the Fourier form of the exact inversion formula, i.e.\ the $n$ one-sided sums
$\sum_{k<h}$ of \cite[eq.~(2.18)]{LRR}, one per level.  A pole of order $n$ at
$\omega=0$ is the Fourier dual of polynomial growth of degree $n-1$ in the
offset: this is exactly the statement of \cite[Thm.~3.4]{LRR} that the
half-strip inverse of a point datum is
$(R^n)^{-1}[\sigma\delta_{h,s}](i,j)=\sigma_j(i-h_j)_+^{\,n-1}$.  In the same
language, the range characterization of \cite{LRR} says that
$g\in\rg(R^n_n)$ (square images) iff the pole is cancelled: the mass
constraints $\sum_hg(h,2t)=\sum_hg(h,2t+1)$ are the conditions
$\hat g(0,2t)=\hat g(0,2t+1)$ that $(\hat g_{2t+1}-\hat g_{2t})/(z-1)$ be regular at
$z=1$ at the top level, and the support constraints at the deeper levels are
the higher-order vanishing conditions at $\omega=0$ that make the iterated
divisions regular, so that $A_\omega^{-1}\hat g(\omega,\cdot)$ is a trigonometric
polynomial in $\omega$ (a finitely supported image).  Out-of-range data are
those for which some division leaves a pole, and the exact inverse then grows
polynomially.

The polynomial filters of Section~\ref{sec:poly} are the other way of handling the same
pole.  Instead of extending the image to the half-strip so that the inverse
exists (at the price of unbounded growth), $p_k(M_k)A_k^*$ replaces
$M_k^{-1}A_k^*=A_k^{-1}$ by a polynomial in $M_k$ with
$0\le p_k(\lambda)\le(1+\varepsilon)/\tau$ on $\operatorname{spec}(M_k)$, instead of $1/\lambda$: the
components in the directions where $M_k$ has eigenvalues $\lesssim(\omega/2)^{2n}$
are damped instead of amplified.  Since a polynomial of degree $d$ in the
Gram $\cR^{n\top}_n \cR^n_n$, whose kernel in the offset is supported in
$|h-h'|\le N-1$, has offset support of width at most $2d(N-1)$ around a point
datum, the filtered inverse has \emph{finite propagation} in the offset, with no
growth of any order, and the information that the single quadrant cannot
recover stably is supplied by the other quadrants through the frame operator
of Section~\ref{sec:oblique}.  The degree $d_k\sim\sqrt{\lambda_{\max}(M_k)/\tau}$ is what a polynomial
needs to resolve the pole down to the scale $\tau$, and the sum
$\sum_kd_k=\cO(N\sqrt{\log N})$ of Theorem~\ref{thm:fast} is the price of the
pole over all offset frequencies.
\section{Oblique inverses}\label{sec:oblique}

\subsection{Left inverses as oblique projections}\label{sec:left}
For $Q$ with $Q^\top\nR$ invertible, $x=(Q^\top\nR)^{-1}Q^\top b$ is the solution of
the Petrov--Galerkin equation $Q^\top(\nR[x]-b)=0$, and $\nR(Q^\top\nR)^{-1}Q^\top$ is
the oblique projector onto $\rg(\nR)$ along $\ker Q^\top$ \cite[Ch.~2]{BenIsrael}.
Every left inverse $L$ arises this way ($Q^\top=L$); $(\nR)^\dagger$ is the one
with $\ker Q^\top=\ker{\nR}^\top$, equivalently the left inverse of minimal norm,
$\|(\nR)^\dagger\|=1/\sigma_{\min}(\nR)$.  For $b=\nR[x]+\nu$ one has
$Lb-x=L\nu$, so we measure a left inverse by its \emph{stability ratio}
\[
\rho(L):=\|L\|_2\,\sigma_{\min}(\nR)\ge1,\qquad \rho\big((\nR)^\dagger\big)=1 .
\]
SPIFE \eqref{eq:spife} and, more generally, the quadrant combinations
\begin{equation}\label{eq:LW}
L_W=\sum_{q}W_q\,T^{n\top}_q(R^n_n)^{-1}(\cdot)_q,\qquad\sum_qW_q=\Id ,
\end{equation}
with $(R^n_n)^{-1}$ the exact inverse \cite{Rim20} and any image-space operators
$W_q$, are explicit left inverses via $L_W\nR=\sum_qW_qT^{n\top}_q(R^n_n)^{-1}R^n_nT^n_q=\sum_qW_q=\Id$, but
$\rho(L_W)\approx10^2,10^5,10^7$ for $N=8,16,32$ even when the $W_q$ are
Fourier cone weights for the cones $\mathcal C_q$: the
unstable directions of $A_k^{-1}$, i.e.\ the eigenvectors of $M_k$ with small
eigenvalues, are not spanned by the plane waves outside $\mathcal C_q$, because
the digital lines are straight lines plus sawtooth deviations of dyadic periods
whose harmonics alias every plane wave into every cone.

\subsection{The spectral partition}\label{sec:partition}
The remedy is to use the \emph{spectral} projector of each quadrant.  For a
threshold $\theta>0$ let $P_k$ be the orthogonal projector onto the eigenvectors
of $M_k$ with eigenvalue $\ge\theta$, and define on the torus $\RR^{L\times L}$
\begin{equation}\label{eq:Phi}
\Phi_q:=T^{n\top}_q\Big[\,\mathrm{DFT}_h^{-1}\,\blockdiag_{k\in[L]}(P_k)\,\mathrm{DFT}_h\Big]T^n_q ,
\qquad F:=\sum_{q=\RN{1}}^{\RN{4}}\Phi_q ,
\end{equation}
where $T^n_q$ now denotes the torus extension of the permutation \eqref{eq:permute}
(reflections $i\mapsto(2^n-1-i)\brem L$ and the transpose) composed with the
restriction to the $N$ columns of the strip, and $T^{n\top}_q$ its adjoint.  Each
$\Phi_q$ is an orthogonal projector and $0\preceq F\preceq4\Id$.

\subsubsection*{Frame operators.}  Recall that a family of closed subspaces
$\mathcal V_1,\dots,\mathcal V_4$ of a Hilbert space $\mathcal H$, with
orthogonal projectors $\Phi_q$, is a \emph{frame} (fusion frame) of $\mathcal H$
with bounds $0<a\le b$ if
\[
a\|x\|^2\le\sum_q\|\Phi_qx\|^2\le b\|x\|^2\qquad\text{for all }x\in\mathcal H,
\]
and the \emph{frame operator} is $F:=\sum_q\Phi_q^*\Phi_q=\sum_q\Phi_q$, a symmetric
positive definite operator with $a\Id\preceq F\preceq b\Id$, and every $x$ is
recovered from its components $\Phi_qx$ by $x=F^{-1}\sum_q\Phi_qx$; the best
bounds are the extreme eigenvalues of $F$, and $b/a=\kappa(F)$ measures how far
the family is from an orthogonal decomposition ($\kappa=1$ when the
$\mathcal V_q$ are mutually orthogonal and span $\mathcal H$).  Here
$\mathcal V_q$ is the stable subspace of quadrant $q$ on the torus and $F$ in
\eqref{eq:Phi} is the frame operator of the four stable subspaces. Its
compression $E^\top FE$ to the square is the frame operator of the subspaces
$\mathcal V_q$ seen from the square, $\langle x,E^\top FEx\rangle=\sum_q\|\Phi_qEx\|^2$,
and the quality of the inverse below is governed by $\kappa(E^\top FE)$.
The operators of quadrant $q$ are written in the same form.  Let
\begin{equation}\label{eq:Mq}
\begin{aligned}
\cR_q&:=\mathrm{DFT}_h^{-1}\blockdiag_{k\in[L]}(A_k)\,\mathrm{DFT}_h\,T^n_q ,\\
M_q&:=\cR_q^\top\cR_q=T^{n\top}_q\,\mathrm{DFT}_h^{-1}\blockdiag_{k\in[L]}(M_k)\,\mathrm{DFT}_h\,T^n_q ,
\end{aligned}
\end{equation}
the cyclic transform of quadrant $q$ acting on torus images (so that
$\iota R^n_n[T^n_qx]=\cR_qEx$ for a square image $x$) and its Gram operator: $M_q$ is
the permuted, offset-Fourier block diagonal operator whose $N\times N$ blocks are
the Gram matrices $M_k=A_k^*A_k$ of Section~\ref{sec:fourier}, one per offset frequency
$\omega_k$.  For a function $r$ on $[0,\infty)$ we write
\[
r(M_q):=T^{n\top}_q\,\mathrm{DFT}_h^{-1}\blockdiag_{k\in[L]}\big(r(M_k)\big)\,\mathrm{DFT}_h\,T^n_q ,
\]
the function $r$ applied to the Gram matrix of each offset frequency (this is
$r$ applied to $\cR_q^\top\cR_q$ on the strip, extended by zero to the rest
of the torus). It is applied in $\cO(N^2\log N)$ plus the cost of the $L$ matrix
functions $r(M_k)$.  In this notation $\Phi_q=\mathbf 1_{[\theta,\infty)}(M_q)$, and the
quadrant-wise cyclic pseudo-inverse is
$\cR^+_q:=T^{n\top}_q\,\mathrm{DFT}_h^{-1}\blockdiag_k(A_k^+)\,\mathrm{DFT}_h$, with
$\cR^+_q=\lambda^{-1}(M_q)\cR_q^\top$ on $\rg(\cR_q)$.

\begin{theorem}[Explicit oblique inverse]\label{thm:main}
For every $\theta>0$ such that $E^\top F E$ is invertible,
\begin{equation}\label{eq:Ltheta}
L^n_\theta:=\big(E^\top FE\big)^{-1}\,E^\top\sum_{q}\Phi_q\,\cR_q^{+}\,\iota
\qquad\text{satisfies}\qquad L^n_\theta\,\nR=\Id .
\end{equation}
Here $E:\RR^{N\times N}\to\RR^{L\times L}$ is the zero padding of a square image
into the torus (so $E^\top$ restricts a torus image to the square), $\iota$ the
embedding of the data $b=(b_q)_q\in\RR^{4\times E^n}$ into the cyclic strips of
Section~3.1 (applied quadrant by quadrant), $\Phi_q$ and $F=\sum_q\Phi_q$ the
projectors and frame operator \eqref{eq:Phi}, and $\cR_q^+$ the quadrant-wise
cyclic pseudo-inverse; thus $E^\top FE$ is the $N^2\times N^2$ symmetric positive
semidefinite compression of $F$ to the square.  Per frequency,
$\Phi_q\cR_q^+$ is the explicit matrix
$V_{\rm sel}\Lambda_{\rm sel}^{-1}V_{\rm sel}^*A_k^*$ formed from the eigenpairs
$(\Lambda_{\rm sel},V_{\rm sel})$ of $M_k$ above $\theta$.  $\nR L^n_\theta$ is the
oblique projector onto $\rg(\nR)$ along $\ker\big(E^\top\sum_q\Phi_q\cR_q^+\iota\big)$,
the complement of the quadrant-wise stable subspaces.
\end{theorem}
\begin{proof}
For $x\in\RR^{N\times N}$ and each quadrant, $\iota\,R^n_n[T^n_qx]=\cR^n_n[T^n_qEx]$
because no line wraps.  Per frequency, $P_kA_k^+A_k=P_k$: $P_k$ projects onto
eigenvectors of $M_k=A_k^*A_k$ with nonzero eigenvalue, which lie in
$\rg(A_k^*)$, and $A_k^+A_k$ is the orthogonal projector onto $\rg(A_k^*)$.  Hence
$\Phi_q\cR_q^+\iota R^n_n[T^n_qx]=\Phi_qEx$; summing over $q$ and restricting to
the square gives $E^\top\sum_q\Phi_q\cR_q^+\iota\,\nR[x]=E^\top FEx$, and
$(E^\top FE)^{-1}$ returns $x$.  The projector statement follows
from $(\nR L^n_\theta)^2=\nR L^n_\theta$ and $\rg(\nR L^n_\theta)=\rg(\nR)$.
\end{proof}

$E^\top FE$ is the compression of a sum of four orthogonal projectors to the
square.  It is invertible for every $\theta$ up to the smallest nonzero
eigenvalue of the $M_k$: then $\Phi_q$ is the orthogonal projector onto
$\rg(\cR_q^\top)$, so $\Phi_qEx=0$ for all $q$ forces $\nR[x]=0$ and hence $x=0$.
For larger $\theta$, Table~\ref{tab:kappa} shows it is positive definite in all
computed cases ($\theta=0.05N$).

\begin{remark}[Filters]\label{rem:filter}
Only $P_kA_k^+A_k=P_k$ was used.  For any function $r$, replacing
$\Phi_q\cR_q^+$ by $r(M_q)\cR_q^\top$ and $F$ by
$F_r:=\sum_qM_q\,r(M_q)$ gives a left inverse, since $r(M_k)A_k^*A_k=M_kr(M_k)$.
The Tikhonov filter $r_\tau(\lambda)=(\lambda+\tau)^{-1} > 0$ gives
$F_\tau=4\Id-\tau\sum_q(M_q+\tau)^{-1}$ and
\begin{equation}\label{eq:Ltau}
\begin{aligned}
L^n_\tau &= \big(E^\top F_\tau E\big)^{-1}E^\top\sum_q(M_q+\tau)^{-1}\cR_q^\top\iota ,
\\
L^n_\tau\nR &= \Id\ \text{for all } \tau > 0,
\qquad
L^n_\tau \to (\nR)^\dagger\ \text{ as } \tau\to\infty,
\end{aligned}
\end{equation}
the limit because $\tau F_\tau\to\sum_qM_q$ and $E^\top\sum_q\cR_q^\top\iota={\nR}^\top$.
The inverse exists for every $\tau>0$: with $\Lambda_{\max}:=\max_k\lambda_{\max}(M_k)$,
$\lambda/(\lambda+\tau)\ge\lambda/(\Lambda_{\max}+\tau)$ on $[0,\Lambda_{\max}]$ gives
$F_\tau\succeq(\Lambda_{\max}+\tau)^{-1}\sum_qM_q$, and $E^\top\big(\sum_qM_q\big)E={\nR}^\top\nR$
is positive definite.
Thus the family $\{L^n_\tau\}$ connects the per-quadrant inverse ($\tau\to0$) to
the Moore--Penrose inverse, with $\kappa(E^\top F_\tau E)$ increasing from $\approx2$ to
$\kappa({\nR}^\top\nR)$.
\end{remark}

\begin{remark}[Relation to the range characterization]
In \cite{LRR}, $\rg(R^n_n)$ is cut out, level by level, by the components of
$T^n_m[g]$ (the partial exact inverse) along the mass and support constraints,
i.e.\ along $\ker{\nm{S}}^\top$.  The projectors $P_k$ remove, at each offset
frequency, exactly the directions in which these level-wise inversions
amplify; Theorem~\ref{thm:main} may be read as: the four quadrants, each
restricted to its stable part of the range, determine the image, and the
frame operator $F$ measures how well.
\end{remark}

\subsection{Stability and conditioning}\label{sec:stab}
Each factor of \eqref{eq:Ltheta} is bounded uniformly in $N$:
$\|\Phi_q\cR_q^+\|\le\theta^{-1/2}$ and $\|F\|\le4$.  What makes $L^n_\theta$ a
good inverse is that $E^\top FE$ is well conditioned (Table~\ref{tab:kappa}): on
the torus itself $F$ has spectrum in $[1,4]$ at $\theta=0.05N$ ($\kappa=2.9,3.3,4.0$
for $N=8,16,32$), and the compression to the square has $\kappa=2.3$--$5.5$ for
$N=8$--$128$.  The stability ratio is $\rho(L^n_\theta)=1.7$--$3$, the error on
white noise a constant $1.8\times$ that of $(\nR)^\dagger$, and for $L^n_\tau$ with
$\tau=N$ it is $1.06$--$1.09\times$.

\begin{table}[h]\centering
\begin{tabular}{l*{5}{S[table-format=2.2]}}
\toprule
$N$ & {8} & {16} & {32} & {64} & {128}\\
\midrule
$\kappa(E^\top FE)$, $\theta=0.05N$ & 2.29 & 2.67 & 3.64 & 4.43 & 5.54\\
error ratio to $(\nR)^\dagger$, $1\%$ white noise & 1.71 & 1.77 & 1.80 & 1.82 & 1.80\\
\midrule
$\kappa(E^\top F_\tau E)$, $\tau=N$ & 3.65 & 4.99 & 6.88 & 9.51 & 13.2\\
error ratio to $(\nR)^\dagger$ & 1.03 & 1.06 & 1.09 & 1.08 & 1.07\\
\bottomrule
\end{tabular}
\caption{Conditioning of the window-compressed frame operator and noise
amplification relative to the Moore--Penrose inverse (dense per-frequency
factors; $L^n_\theta\nR=\Id$ holds to $10^{-15}$).}
\label{tab:kappa}
\end{table}

Because $\kappa(E^\top FE)$ is small, $(E^\top FE)^{-1}$ may be replaced by a
Chebyshev polynomial $c_D(E^\top FE)$ of fixed degree $D$ on
$[\lambda_{\min},\lambda_{\max}]$, with relative error
$\le2\big((\sqrt\kappa-1)/(\sqrt\kappa+1)\big)^D$ (see \cite{GolubVarga61,Saad}). For example, $D=40$
gives $10^{-14}$ for $\kappa\le5$.  With this replacement \eqref{eq:Ltheta} is
a closed-form expression in which no iteration is run to a tolerance.
\section{Polynomial filters and their cost}\label{sec:poly}

With dense eigendecompositions the cost of \eqref{eq:Ltheta} is $\cO(N^4)$ setup
and $\cO(N^3)$ per application.  Remark~\ref{rem:filter} removes both; the
polynomial filters are shown in Figures~\ref{fig:filters}--\ref{fig:response}.

\begin{definition}[Polynomial-filter inverse]
Let $\tau>0$; for each $k\in[L]$ let $\Lambda_k\ge\lambda_{\max}(M_k)$ and let $p_k$
be the degree-$d_k$ Chebyshev polynomial approximating $(\lambda+\tau)^{-1}$ on
$[0,\Lambda_k]$ ($p_k(M_k)y$ is the result of $d_k$ steps of the Chebyshev
semi-iteration \cite{GolubVarga61,Saad} for $(M_k+\tau)z=y$ started at $0$).  With
$r(M_k):=p_k(M_k)$ in Remark~\ref{rem:filter},
\begin{equation}\label{eq:Lp}
L^n_p:=\big(E^\top F_pE\big)^{-1}E^\top\sum_qp(M_q)\cR_q^\top\iota ,\qquad F_p:=\sum_qM_q\,p(M_q).
\end{equation}
Written out with \eqref{eq:Mq},
\begin{align*}
F_p&=\sum_{q=\RN{1}}^{\RN{4}}T^{n\top}_q\,\mathrm{DFT}_h^{-1}\blockdiag_{k\in[L]}\big(M_kp_k(M_k)\big)\,\mathrm{DFT}_h\,T^n_q ,\\
p(M_q)\cR_q^\top&=T^{n\top}_q\,\mathrm{DFT}_h^{-1}\blockdiag_{k\in[L]}\big(p_k(M_k)A_k^*\big)\,\mathrm{DFT}_h ,
\end{align*}
so that $E^\top F_pEx$ is computed by padding $x$ to the torus and, for each
quadrant, permuting, taking the DFT in the offset, applying
$A_k^*A_k$ after $p_k(M_k)$ ($d_k$ Chebyshev steps, each one product $A_k^*A_k$
by the butterflies) at every frequency, transforming back, permuting back and
summing, and finally restricting to the square; $\sum_qp(M_q)\cR_q^\top\iota b$
is computed in the same way from the data embedded in the strips, with $A_k^*$
in place of $A_k^*A_k$.
\end{definition}

\begin{theorem}\label{thm:fast}
\mbox{}
\begin{enumerate}
\item[(i)] $L^n_p\nR=\Id$ exactly whenever $E^\top F_pE$ is invertible; in
particular for every choice of degrees $d_k\ge1$ and bounds
$\Lambda_k\ge\lambda_{\max}(M_k)$, e.g.\ the certified bounds \eqref{eq:cert}.
\item[(ii)] One application of $\sum_qp(M_q)\cR_q^\top\iota$, or of $F_p$, costs
$\cO\big(\sum_{k\in[L]}d_k\cdot N\log N\big)$ operations.
\item[(iii)] If $\Lambda_k\le C(N)N/|\omega_k|$ for $k\ne0$ and $\Lambda_0=N^2$, then
every $d_k\ge\tfrac12\sqrt{1+\Lambda_k/\tau}\,\ln(2/\varepsilon)$ gives
\[
\max_{0\le\lambda\le\Lambda_k}|p_k(\lambda)-(\lambda+\tau)^{-1}|\le\frac{\varepsilon}{\tau}
\]
and, for the smallest such integers and $\tau=cN$,
\[
\sum_kd_k=\cO\Big(N\sqrt{\frac{C(N)}{c}}\,\log\frac1\varepsilon\Big).
\]
For the explicit bounds
\eqref{eq:Lambda}, $C(N)=2(1+\log_2N)$, and one application of $F_p$ or of
$\sum_qp(M_q)\cR_q^\top\iota$ costs $\cO(N^2\log^{3/2}N\log(1/\varepsilon))$,
unconditionally.  If moreover $\Lambda_k\ge\lambda_{\max}(M_k)$, then
$\|p_k(M_k)-(M_k+\tau)^{-1}\|\le\varepsilon/\tau$.
\item[(iv)] Realize $(E^\top F_pE)^{-1}$ by a Chebyshev polynomial of
degree $D=\cO(\sqrt{\kappa}\log(1/\varepsilon))$, $\kappa:=\kappa(E^\top F_pE)$.  Then
the cost of $L^n_pb$ is
\[
\cO\Big(N^2\log^{3/2}N\cdot\sqrt{\kappa(E^\top F_pE)}\,\log\frac{1}{\varepsilon_{\rm in}}\log\frac1\varepsilon\Big),
\]
where $\varepsilon_{\rm in}$ is the tolerance of the inner filters in (iii) and
$\varepsilon$ that of the outer polynomial, with
$\kappa(E^\top F_pE)\approx\tau/\lambda_{\min}(R^{n\top}R^n)$ for large $N$
(Section~\ref{sec:cond}).  All degrees are fixed in a data-independent setup of cost
$\cO(N^2\log N)$.
\end{enumerate}
\end{theorem}
\begin{proof}
(i) is Remark~\ref{rem:filter} with $r=p_k$ at frequency $k$; the same
polynomial enters $p(M_q)\cR_q^\top$ and $F_p$ by construction.  For
invertibility, write $\mu=\lambda+\tau$ and $r_{d}(\mu)=1-\mu\,p_k(\lambda)$ for the
residual polynomial of the $d=d_k$ Chebyshev steps.  If
$\operatorname{spec}(M_k)\subset[0,\Lambda_k]$, then $\mu\in[\tau,\Lambda_k+\tau]$, where
$|r_d(\mu)|\le1/T_d\big((\kappa_k+1)/(\kappa_k-1)\big)<1$ for $d\ge1$
($\kappa_k=1+\Lambda_k/\tau$, $T_d$ the Chebyshev polynomial), so
$p_k(\lambda)=(1-r_d(\mu))/\mu>0$ on $\operatorname{spec}(M_k)$.  With $m>0$ the
minimum of the $p_k$ over these spectra, $F_p\succeq m\sum_qM_q$ and
$E^\top F_pE\succeq m\,{\nR}^\top\nR$, which is positive definite.  (With $d_k=0$,
$p_k=0$, and with $\Lambda_k<\lambda_{\max}(M_k)$, $p_k$ may change sign on the
spectrum.)  (ii) Each
Chebyshev step at frequency $k$ is one product $A_k^*A_k$ with a vector of
length $N$, i.e.\ $\cO(N\log N)$ by \eqref{eq:butterfly}, and the frequencies
decouple; add $\cO(N^2\log N)$ for the DFTs and the final $A_k^*$.  (iii) The
Chebyshev iteration for spectrum in $[\tau,\Lambda_k+\tau]$ has error
$\le2\big((\sqrt{\kappa_k}-1)/(\sqrt{\kappa_k}+1)\big)^d$, $\kappa_k=1+\Lambda_k/\tau$
\cite[\S12.3]{Saad} for the residual $r_d(\mu)=1-\mu\,p_k(\lambda)$, and
$|p_k(\lambda)-(\lambda+\tau)^{-1}|=|r_d(\mu)|/\mu\le|r_d(\mu)|/\tau$ since $\mu=\lambda+\tau\ge\tau$.
As $\ln\frac{\sqrt{\kappa_k}+1}{\sqrt{\kappa_k}-1}=2\operatorname{artanh}(\kappa_k^{-1/2})\ge2\kappa_k^{-1/2}$,
the error is $\le\varepsilon/\tau$ once $d\ge\frac12\sqrt{\kappa_k}\ln(2/\varepsilon)$; with $|\omega_k|=2\pi\min(k,L-k)/L$,
$\sum_{k\ne0}\sqrt{\Lambda_k/\tau}\le2\sum_{k=1}^{N}\sqrt{C(N)NL/(2\pi kcN)}=\cO(N\sqrt{C(N)/c})$,
and $d_0=\cO(\sqrt{N/c}\log(1/\varepsilon))$; with \eqref{eq:Lambda}, $C(N)=2(1+\log_2N)$
by definition, and when $\mathrm{spec}(M_k)\subset[0,\Lambda_k]$ the bound on the
polynomial is a bound on the operator.  (iv) Multiply by $D$ and by the four
quadrants.
\end{proof}
By (i), $\varepsilon_{\rm in}$ does not affect the left-inverse identity, only the
oblique projection, and may be fixed independently of $\varepsilon$.
Section~\ref{sec:inner} shows that $\varepsilon_{\rm in}=0.1$ leaves the noise
amplification unchanged, so that in practice the cost of $L^n_pb$ is
$\cO(N^2\log^{3/2}N\sqrt{\kappa(E^\top F_pE)}\log(1/\varepsilon))$.

\begin{figure}[t]\centering
\includegraphics[width=\textwidth]{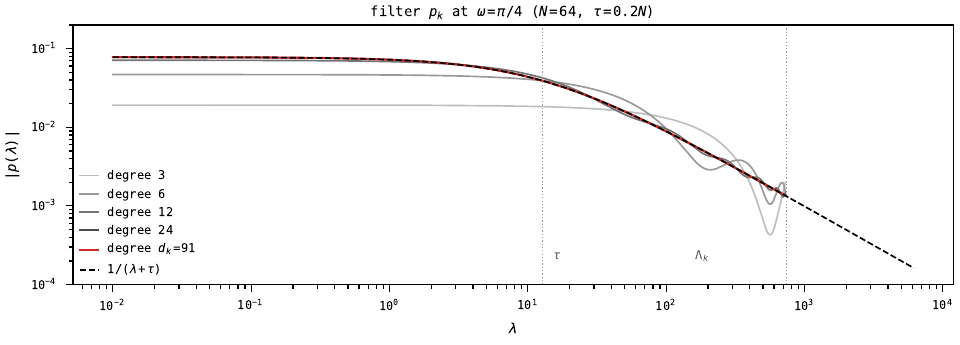}
\caption{The Chebyshev filter $p_k$ at $\omega_k=\pi/4$ ($N=64$,
$\tau=0.2N$) for increasing degree, converging to $(\lambda+\tau)^{-1}$ on
$[0,\Lambda_k]$; the final degree $d_k$ is fixed by $\varepsilon=10^{-10}$.}
\label{fig:filters}
\end{figure}
\begin{figure}[t]\centering
\includegraphics[width=\textwidth]{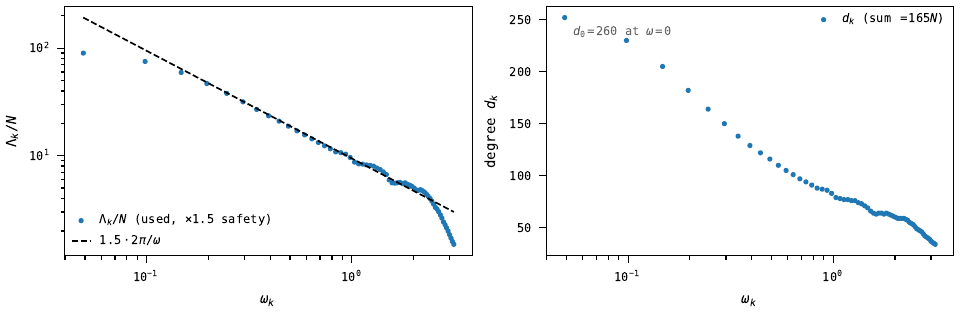}
\caption{The spectral bounds $\Lambda_k$ (power iteration with safety factor
$1.5$) follow $2\pi N/|\omega_k|$ away from $\omega=0,\pi$ (left), and the resulting
degrees $d_k\propto\sqrt{1+\Lambda_k/\tau}$ sum to $165N$ (right); $N=64$, $\tau=0.2N$.}
\label{fig:degrees}
\end{figure}
\begin{figure}[t]\centering
\includegraphics[width=\textwidth]{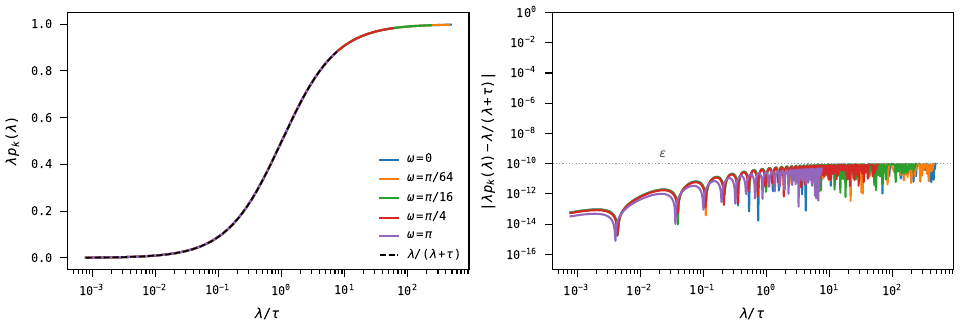}
\caption{The filter factors $\lambda p_k(\lambda)$ entering $F_p=\sum_qM_qp(M_q)$
(left) and their deviation from the Tikhonov response $\lambda/(\lambda+\tau)$
(right), for five offset frequencies: all are at the level $\varepsilon=10^{-10}$ over
the whole spectrum $[0,\Lambda_k]$, so $F_p$ differs from $F_\tau$ by $\cO(\varepsilon)$
while $L^n_p$ remains an exact left inverse (Theorem~\ref{thm:fast}(i)).}
\label{fig:response}
\end{figure}

\subsection{Explicit and certified spectral bounds}\label{sec:bounds}
\subsubsection*{The entries of $M_\omega$.}  The bounds $\Lambda_k$ can be chosen explicitly because the entries of $M_\omega$ have a closed form.
\begin{lemma}[The Gram kernel]\label{lem:kernel}
For $t,t'\in[N]$ let $\alpha_m=\lfloor t'/2^{n-m}\rfloor-\lfloor t/2^{n-m}\rfloor$,
$m=0,\dots,n$ (so $\alpha_0=0$, $\alpha_n=t'-t$), and $e_m=\alpha_{m+1}-\alpha_m$.  Then
\[
|M_\omega(t,t')|=\prod_{m=0}^{n-1}2\,\Big|\cos\frac{\omega e_m}{2}\Big| .
\]
\end{lemma}
\begin{proof}
For an integer $\alpha$ let $M_{\omega,\alpha}=\sum_{s\in[N]}z^{\alpha s}a_sa_s^*$ with
$a_s(t)=z^{\delta_s(t)}$, $z=e^{\mathrm i\omega}$, so that $M_\omega=M_{\omega,0}$.  The
weights of Section~2 satisfy $w_{n-1}(s)=\lceil s/2\rceil$ and
$w_j(s)=w'_j(\lfloor s/2\rfloor)$ for $j<n-1$, with $w'$ the weights of $n-1$ levels,
i.e.\ $\delta_s(t)=\lceil s/2\rceil\,t_{n-1}+\delta'_{\lfloor s/2\rfloor}(t\bmod2^{n-1})$.
Writing $s=2s'+s_0$ and summing over $s_0\in\{0,1\}$, the block of $M_{\omega,\alpha}$
with top bits $(b,b')$ of $(t,t')$ equals $(1+z^{\alpha+b'-b})\,M'_{\omega,2\alpha+b'-b}$
on the lower bits, with $M'$ the matrices of $n-1$ levels.  Iterating from the
top bit, $\alpha$ runs through $\alpha_0=0$, $\alpha_{m+1}=2\alpha_m+(b'_m-b_m)$, which
is the stated sequence, the factor at step $m$ is $1+z^{\alpha_{m+1}-\alpha_m}$, and
$|1+z^e|=2|\cos(\omega e/2)|$.
\end{proof}
The lemma has a simple reading.  The integer $e_m$ counts the odd multiples
of $2^{n-m-1}$ between $t$ and $t'$, that is, the number of times the two
columns fall on different sides of a dyadic boundary at level $m$; each such
level contributes a factor $2|\cos(\omega e_m/2)|$, between $0$ and $2$.  Two
columns in the same half, quarter, \dots\ of the strip share their top
levels, where $e_m=0$ and the factor is the maximal value $2$.

\subsubsection*{An explicit choice of $\Lambda_k$.}
For straight lines the exponents double from level to level,
$e_m=(t'-t)/2^{n-m}$, and the product telescopes to the Dirichlet kernel
\[
|M_\omega(t,t')|=\Big|\frac{\sin(\omega d/2)}{\sin(\omega d/2N)}\Big|\le\min\Big(N,\frac{2N}{\omega|d|}\Big),
\qquad d=t'-t .
\]
The largest eigenvalue is at most the largest row sum,
$\lambda_{\max}(M_\omega)\le\max_t\sum_{t'}|M_\omega(t,t')|$, and summing the Dirichlet
bound over $|d|<N$ gives $\lambda_{\max}\lesssim\frac{2N}{\omega}(1+\log N)$.  The
factor $\log N$ is the same growth $C(N)\sim\log N$ that power iteration finds
for the digital lines (Section~\ref{sec:fourier}).  We therefore take
\begin{equation}\label{eq:Lambda}
\Lambda_0:=N^2,\qquad \Lambda_k:=\min\Big\{N^2,\ \frac{2N(1+\log N)}{|\omega_k|}\Big\}\quad(k\ne0).
\end{equation}

How good is this choice?  At $\omega=0$ it is exact, since $\lambda_{\max}(M_0)=N^2$.
At every other frequency and for all $N\le512$ it lies above the measured
$\lambda_{\max}(M_k)$, by a factor of at least $1.7$ and typically $1.9$--$2.6$.  So
\eqref{eq:Lambda} is a safe bound with room to spare, but not a tight one, and
the overestimate costs inner degrees: since $d_k$ grows like
$\sqrt{\Lambda_k}$, the degree sums are $173N$, $196N$, $217N$ for $N=32,64,128$,
against $154N$, $165N$, $175N$ when $\Lambda_k$ is estimated by power iteration.
The explicit bounds thus cost $12$--$24\%$ more inner steps and save the
power iterations of the setup; the left-inverse error and the noise
amplification are unchanged.

Why is \eqref{eq:Lambda} not a theorem?  The straight-line argument does not
carry over, because for digital lines the sequence $e_m$ need not double.  A
single entry shows the effect.  At $\omega=2\pi/3$, $t=0$, $t'=3N/4$ one finds
$e_m=1,2,3,6,12,\dots$; from the third level on every $\omega e_m/2$ is a multiple
of $\pi$, every cosine is $\pm1$, and $|M_\omega(0,3N/4)|=N/4$, where the Dirichlet
kernel of a straight line is about $2$.  Such large off-diagonal entries make
the row sums grow faster than \eqref{eq:Lambda}, by roughly one more factor of
$\log N$, as the first row of Table~\ref{tab:lambda} shows: the ratio passes
$1$ at $N=32$ and reaches $2.6$ at $N=1024$.  Yet the eigenvalues do not follow
the row sums: $\lambda_{\max}(M_\omega)$ stays below $0.65$ times \eqref{eq:Lambda}
throughout.  The large entries must therefore cancel in the eigenvalue, and
a proof that \eqref{eq:Lambda} majorizes $\lambda_{\max}(M_k)$ would have to capture
that cancellation, which the row-sum bound discards by construction.  We
leave this open.  Nothing about exactness rests on it in practice: by
Theorem~\ref{thm:fast}(i) the identity $L^n_p\nR=\Id$ needs only an invertible
$E^\top F_pE$, which the bound $\Lambda_k\ge\lambda_{\max}(M_k)$ guarantees and which
holds in all computed cases; what the bound controls is the accuracy of the
filters.

\begin{table}[t]\centering
\begin{tabular}{l*{8}{S[table-format=1.2]}}
\toprule
$N$ & {8} & {16} & {32} & {64} & {128} & {256} & {512} & {1024}\\
\midrule
largest row sum of $M_\omega$ / \eqref{eq:Lambda} & 0.83 & 0.88 & 1.01 & 1.18 & 1.41 & 1.72 & 2.10 & 2.63\\
$\Lambda^{\rm cert}_k/\lambda_{\max}(M_k)$, median over $k$ & 1.3 & 1.7 & 2.2 & 3.0 & 4.1 & 5.3 & 6.9 & 8.9\\
$\sum_kd_k$ with $\Lambda^{\rm cert}_k$ / with \eqref{eq:Lambda} & 0.90 & 0.94 & 1.01 & 1.09 & 1.20 & 1.33 & 1.49 & 1.68\\
\bottomrule
\end{tabular}
\caption{Explicit and certified spectral bounds.  First row: the largest
row sum $\max_t\sum_{t'}|M_\omega(t,t')|$, maximized over $\omega$, relative to
\eqref{eq:Lambda}.  Second row: overestimate of the certified bound
\eqref{eq:cert}.  Third row: its cost in inner degrees ($\tau=0.2N$).}
\label{tab:lambda}
\end{table}

\subsubsection*{A certified upper bound.}
If a proved bound is wanted, it can be computed at negligible cost.  The
row-sum bound $\lambda_{\max}(M_k)\le\max_t\sum_{t'}|M_k(t,t')|$ is rigorous, and the
block recursion in the proof of Lemma~\ref{lem:kernel} bounds the row sums
level by level, without forming $M_k$.  Let
$r_m(\alpha):=\max_t\sum_{t'}|M^{(m)}_{\omega,\alpha}(t,t')|$ for the $m$-level matrices
with twist $\alpha$.  Then $r_0\equiv1$ and
\begin{equation}\label{eq:cert}
r_m(\alpha)\ \le\ |1+z^\alpha|\,r_{m-1}(2\alpha)+\max_\pm|1+z^{\alpha\pm1}|\,r_{m-1}(2\alpha\pm1),
\qquad \Lambda^{\rm cert}_k:=r_n(0)\ \ge\ \lambda_{\max}(M_k).
\end{equation}
Level $m$ needs only the twists $|\alpha|\le2^{n-m}-1$, so all $L$ bounds together
cost $\cO(N^2)$ operations, $0.07$\,s at $N=1024$ on one core.

The price of certification is read off the last two rows of
Table~\ref{tab:lambda}.  Being a row-sum bound, $\Lambda^{\rm cert}_k$ inherits the
extra $\log N$ of the first row: it overestimates $\lambda_{\max}(M_k)$ by a factor
growing from $1.3$ to $8.9$ over $N=8,\dots,1024$, and its ratio to
\eqref{eq:Lambda} has median between $0.7$ and $3.7$, below $1$ only for the
smallest $N$.  Because the degrees grow only like
the square root of the bound, the effect on the work is much milder: the
degree sums are at most $1.68$ times those of \eqref{eq:Lambda} up to $N=1024$,
and even slightly smaller for $N\le16$.  Asymptotically the certified sum is
$\sum_kd_k=\cO(N\log N\log\varepsilon^{-1})$ instead of
$\cO(N\sqrt{\log N}\log\varepsilon^{-1})$.

In return, every statement about the filters in Theorem~\ref{thm:fast} holds
without any measurement.  Since $\mathrm{spec}(M_k)\subset[0,\Lambda^{\rm cert}_k]$, one has
$\|p_k(M_k)-(M_k+\tau)^{-1}\|\le\varepsilon/\tau$, $F_p\succeq0$, $E^\top F_pE$ is invertible,
and $\lambda_{\max}(E^\top F_pE)\le4(1+\varepsilon)$.  The only quantity left to estimate
is $\lambda_{\min}(E^\top F_pE)$, which fixes the outer degree $D$; it is the subject
of the rest of this section.

\subsubsection*{The lower end of the outer interval.}
The upper end is explicit: with \eqref{eq:Lambda} or \eqref{eq:cert}, each
factor satisfies $\lambda p_k(\lambda)\le1+\varepsilon$, so
$\lambda_{\max}(E^\top F_pE)\le4(1+\varepsilon)$.  The lower end must be
\emph{under}estimated, since a Chebyshev polynomial built on too small an
interval diverges outside it.  Three statements can be made: a proved but
pessimistic bound, a computed and guaranteed one, and an empirical formula.
\begin{proposition}[A lower bound]\label{prop:lower}
With $\Lambda_k\ge\lambda_{\max}(M_k)$, the Tikhonov frame operator satisfies
\[
E^\top F_\tau E\ \succeq\ R^{n\top}\,W\,R^n ,
\]
where on each quadrant $W=\mathrm{DFT}_h^{-1}\,\mathrm{diag}_k\big((\Lambda_k+\tau)^{-1}\big)\,\mathrm{DFT}_h$,
an explicit ramp-type weighting of the data ($w(\omega_k)\approx|\omega_k|/2N(1+\log N)$ for
\eqref{eq:Lambda}); hence $\lambda_{\min}(E^\top F_\tau E)\ge\lambda_{\min}(R^{n\top}WR^n)$.
\end{proposition}
\begin{proof}
$\lambda/(\lambda+\tau)\ge\lambda/(\Lambda_k+\tau)$ on $[0,\Lambda_k]$, so per frequency
$M_k(M_k+\tau)^{-1}\succeq M_k/(\Lambda_k+\tau)$; summing over quadrants and
compressing to the square, $\sum_qE^\top T_q^\top\mathrm{DFT}_h^{-1}\blockdiag_k(M_k)\mathrm{DFT}_hT_qE=R^{n\top}R^n$
by the no-wrap property, and the weights ride along.
\end{proof}
\emph{A proved bound.}  The proposition reduces the lower frame bound to the
conditioning of the ramp-weighted normal operator $R^{n\top}WR^n$, which is where
the genuine ill-conditioning of the problem sits.  It is valid but
pessimistic.  With the measured $\lambda_{\min}(R^{n\top}\Lambda R^n)\sim N^{0.6}$
($\Lambda$ the ramp) it gives $\lambda_{\min}(E^\top F_\tau E)\gtrsim N^{-0.4}/\log N$,
about $0.03$ at $N=128$, where the true value is $0.70$.  Since the outer
degree grows like $\lambda_{\min}^{-1/2}$, building $D$ on this bound would make it
about five times too large at $N=128$, and $\cO(N^{0.2})$ instead of nearly
constant.

\emph{A guaranteed computable bound.}  A few Lanczos steps on $E^\top F_pE$ give
an enclosure that is both rigorous and nearly sharp.  If $\theta$ is a Ritz
value with Ritz vector $u$ and residual $r=(E^\top F_pE-\theta)u$, then
$\lambda_{\min}\ge\theta-\|r\|$; likewise $\lambda_{\max}\le\theta'+\|r'\|$ for the top
Ritz pair.  The computation is data independent and costs
$\cO(N^2\log N)$ per step.  In the implementation it replaces any safety
factor, and the enclosures are within $10^{-3}$ of the converged eigenvalues
($N\le64$), so nothing is lost to pessimism.

\emph{An empirical formula.}  For a setup with no computation at all, one
can fit the measured values in Table~\ref{tab:lammin}.  At $\tau=0.2N$ they lie
above $4/\log N$ for every $N\le256$, but the margin is thin: between $6\%$ and
$28\%$, and only $6\%$ at $N=256$.  Moreover from $N=128$ to $256$ the measured
value drops by a factor $0.76$, while $4/\log N$ drops only by $0.88$; if this
trend continues, the two cross near $N=512$.  So $4/\log N$ should not be
relied upon.  The safer choice $\lambda_{\min}^{\rm expl}=2/\log N$, with
$\lambda_{\max}^{\rm expl}=4.2$ and \eqref{eq:Lambda}, makes the setup entirely
computation-free.  The price is in the outer degree: $D=31,35,40,43$ for
$N=8,\dots,64$ instead of $19$--$26$ with the Lanczos enclosure, about $1.6$
times more outer steps, each a full application of $F_p$.  The left-inverse
error ($\le2\cdot10^{-11}$) and the error ratios are unchanged.

The larger regularization $\tau=N$ behaves differently.  Its values decay
faster, roughly like $N^{-0.4}$, and reach $0.21$ at $N=256$, so that
$\kappa(E^\top F_\tau E)=18.7$ there and the outer degree grows like $N^{0.2}$.
This is the trade-off of $\tau$: a larger $\tau$ brings $L^n_\tau$ closer to the
Moore--Penrose inverse (noise amplification $1.07$ instead of $1.36$ times
that of $(\nR)^\dagger$, Section~\ref{sec:num}) at the cost of a worse
conditioned frame operator.  Whether $\lambda_{\min}(E^\top F_\tau E)\gtrsim1/\log N$
holds for fixed $\tau/N$ is open; a proof would make the outer degree
$\cO(\sqrt{\log N})$ with explicit constants.

\begin{table}[t]\centering
\begin{tabular}{l*{6}{S[table-format=1.2]}}
\toprule
$N$ & {8} & {16} & {32} & {64} & {128} & {256}\\
\midrule
$\lambda_{\min}(E^\top F_pE)$, $\tau=0.2N$ & 1.44 & 1.23 & 1.02 & 0.85 & 0.70 & 0.53\\
$4/\log N$ & 1.33 & 1.00 & 0.80 & 0.67 & 0.57 & 0.50\\
$\lambda_{\min}(E^\top F_\tau E)$, $\tau=N$ & 0.90 & 0.71 & 0.54 & 0.40 & 0.30 & 0.21\\
\bottomrule
\end{tabular}
\caption{The lower end of the outer interval: measured smallest eigenvalues
of the frame operator, against the candidate formula $4/\log N$.}
\label{tab:lammin}
\end{table}

\begin{algorithm}[h]
\caption{Explicit oblique inverse $x=L^n_pb$}
\begin{algorithmic}[1]
\Statex \textbf{Setup} (once per $N$): take $\Lambda_k$ from \eqref{eq:Lambda} (or
estimate it by power iteration on $M_k$ with a safety factor); set $\tau=cN$ and
$d_k$; estimate $[\lambda_{\min},\lambda_{\max}]$ of $E^\top F_pE$ by a few Lanczos
steps; set $D$.
\Statex \textbf{Apply:}
\State For $q=\RN{1},\dots,\RN{4}$: embed $b_q$ cyclically ($\iota$), DFT in the
offset, apply $A_k^*$ and then $p_k(M_k)$ ($d_k$ Chebyshev steps; rows are
processed only while $d_k>$ step), inverse DFT, apply $T^{n\top}_q$; sum over $q$
and restrict to the square: $y=E^\top\sum_qp(M_q)\cR_q^\top\iota b$.
\State $x\gets c_D(E^\top F_pE)\,y$ by $D$ Chebyshev steps, each applying $F_p$ to
an image: pad, and for each $q$ DFT, $p_k(M_k)$, $A_k^*A_k$, inverse DFT,
$T^{n\top}_q$; sum; restrict.
\end{algorithmic}
\end{algorithm}

\section{A hybrid inverse: explicit cross-quadrant level, filtered upper levels}\label{sec:hyb}
The polynomial filters of Section~\ref{sec:poly} treat the $n$ levels of each
quadrant alike.  The cross-quadrant level $\Sone$, however, has an explicit
decomposition \cite[Thm.~4.1, 5.1]{LORR} and, as Section~\ref{sec:g1} shows, an
explicit Gram inverse; its anti-phase singular values $2\sin\frac\omega4$ are
small for a single quadrant but harmless once the four quadrants are fused.
This suggests inverting level 1 exactly and filtering only the
single-quadrant levels $m=2,\dots,n$.  Write
\[
Q:=S^n_n\cdots S^n_2,\qquad \nR=(\Id_4\otimes Q)\,\Sone ,
\]
and, in the offset-Fourier domain of Section~\ref{sec:fourier},
\[
A_k=A'_k\,\mathcal B^n_1,\qquad A'_k:=\mathcal B^n_n\cdots\mathcal B^n_2 ,
\qquad M'_k:=A_k'^*A'_k=\text{Gram of the upper levels},
\]
with $p(M')$ and $M'p(M')$ acting per quadrant and offset frequency through
$p_k(M'_k)$ and $M'_kp_k(M'_k)$ as in \eqref{eq:Mq}.  SPIFE is
${\Sone}^\dagger\prod_{m\ge2}{\npmz{S}{m}}^\dagger$, a left inverse of $\nR$ which is
unstable because the single-quadrant levels have singular values
$2\sin\frac{\omega}4\approx\omega/2$ in their anti-phase directions
(Proposition~\ref{prop:pole}); Sections~\ref{sec:oblique} and \ref{sec:poly}
replace all level inverses by the per-quadrant filter $p(M_q)$ and fuse the
quadrants through $E^\top F_pE=\sum_qE^\top T_q^\top M_qp(M_q)T_qE$.

\subsection{The hybrid inverse}
The proposal is to keep SPIFE's explicit inverse of the cross-quadrant level
and to apply the polynomial filter to the upper levels only.  The stable way
to arrange this is to backproject the data through the upper levels with the
filter, to level-1 data, and to fuse the quadrants there.  We write
\begin{equation}\label{eq:hybnot}
\Phi':=\Id_4\otimes M'p(M'),\qquad G_1:={\Sone}^\top\Sone,\qquad G_h:={\Sone}^\top\Phi'\,\Sone ,
\end{equation}
for the filtered upper-level Gram matrices, the level-1 Gram matrix and their
combination; here $p(M')$ and $M'p(M')$ act per quadrant and offset frequency
through $p_k(M'_k)$ and $M'_kp_k(M'_k)$ ($d_k$ Chebyshev steps with the butterflies
of the upper levels), and ${\Sone}^\dagger=G_1^{-1}{\Sone}^\top$ is SPIFE's
cross-quadrant inverse.  The hybrid inverse is
\begin{equation}\label{eq:hyb}
L_{\rm hyb}\,b:=\widetilde G^{-1}{\Sone}^\dagger\,Y,\qquad
Y:=\big(\Id_4\otimes p(M')Q^\top\big)\iota b,\qquad
\widetilde G:={\Sone}^\dagger\,\Phi'\,\Sone .
\end{equation}

\begin{proposition}[The hybrid is an exact left inverse]\label{prop:hyb}
For every family of filter polynomials $p_k$:
\begin{enumerate}
\item[(i)] $L_{\rm hyb}\,b=G_h^{-1}{\Sone}^\top Y$, and $L_{\rm hyb}\nR=\Id$;
\item[(ii)] $\widetilde G=G_1^{-1}G_h$ has real spectrum in
$[0,\max_\lambda\lambda p(\lambda)]\subset[0,1+\varepsilon]$.
\end{enumerate}
\end{proposition}
\begin{proof}
(i) Since ${\Sone}^\dagger=G_1^{-1}{\Sone}^\top$, we have $\widetilde G=G_1^{-1}G_h$ and
$\widetilde G^{-1}{\Sone}^\dagger=G_h^{-1}{\Sone}^\top$.  For $b=\nR x=(\Id\otimes Q)\Sone x$,
\[
{\Sone}^\top Y={\Sone}^\top\big(\Id\otimes p(M')Q^\top Q\big)\Sone x=G_hx ,
\]
because the cyclic transforms agree with the true ones on padded square
images, as in the proof of Theorem~\ref{thm:main}; hence $L_{\rm hyb}\nR x=x$.
(ii) With the thin singular value decomposition $\Sone=\widehat U\widehat\Sigma\widehat V^\top$,
\[
G_1^{-1}G_h=\widehat V\widehat\Sigma^{-1}\big[\widehat U^\top\Phi'\,\widehat U\big]\widehat\Sigma\widehat V^\top ,
\]
which is similar to the compression $\widehat U^\top\Phi'\,\widehat U$ of $\Phi'$ to
$\rg\Sone$; its eigenvalues lie in the range of $\lambda p(\lambda)$ on the spectra
of the $M'_k$.
\end{proof}

Thus the hybrid is exactly the construction of Section~\ref{sec:poly} with two changes:
the filter is applied to the Gram matrices $M'_k$ of levels $2,\dots,n$
instead of $M_k$, the fusion over quadrants being carried by ${\Sone}^\top$
instead of the permutations $T_q$; and the explicit $G_1^{-1}$ is
used as a preconditioner for the outer solve, so that the outer Chebyshev
polynomial is built on the spectrum of $G_1^{-1}G_h$ rather than that of
$G_h$.  Section~\ref{sec:g1} makes $G_1$ and $G_1^{-1}$ explicit: $G_1$ is a
nearest-neighbour stencil with spectrum in $[4,16]$, so $G_1^{-1}$ is applied
to any fixed accuracy by a short Chebyshev polynomial at $\cO(N^2)$ cost, and
the level-1 pseudo-inverse of \cite[Thm.~5.1]{LORR} is the inverse of a
truncated Gram $4\Id+8\Pi$, also $\cO(N^2)$.  Each outer step therefore costs
one application of $G_h$ (the upper-level butterflies, $\sum_kd_k$ steps) plus
$\cO(N^2)$.  The reconstruction does not depend on the preconditioner; only the
degree of the outer polynomial does.
\subsection{The level-1 Gram and its explicit relatives}\label{sec:g1}
Write the image as $x\in\RR^{N\times N}$, let $K=\mathrm{tridiag}(1,2,1)\in\RR^{N\times N}$
and let $J=\Id_{N/2}\otimes\begin{psmallmatrix}0&1\\1&0\end{psmallmatrix}$ swap
the two members of each pair $(2a,2a+1)$.

\begin{lemma}[Level-1 Gram]\label{lem:g1}
$G_1={\Sone}^\top\Sone=8\Id+K\otimes J+J\otimes K$, i.e.\ $G_1x=8x+KxJ+JxK$, and
$\operatorname{spec}(G_1)\subset[4,16]$ for every $N$.
\end{lemma}
\begin{proof}
For quadrant \RN{3} ($T=\Id$) the level-1 operator on a column pair is
$g(h,2\ell)=x(h,2\ell)+x(h,2\ell+1)$, $g(h,2\ell+1)=x(h,2\ell)+x(h+1,2\ell+1)$, so
${S^n_1}^\top S^n_1$ couples $x(h,2\ell)$ with $x(h,2\ell+1)$ and $x(h\pm1,2\ell+1)$ with
unit weights and has diagonal $2$; quadrant \RN{2} is the same with the row
order reversed, and the sum of the two is $4\Id+K\otimes J$.  Quadrants \RN{1}
and \RN{4} give $4\Id+J\otimes K$ with the roles of rows and columns exchanged.

For the enclosure, $K$ with zero boundary conditions is the compression of the
circulant $K_c$ on a period $N'\ge N+2$ (even), and $J$ acts inside pairs, so
$G_1$ is the compression of the periodic operator $G_c=8\Id+K_c\otimes J+J\otimes K_c$
and its Rayleigh quotients lie in the convex hull of $\operatorname{spec}(G_c)$.
On the pairs $(2a,2a+1)$ with pair frequency $\alpha$, $K_c$ has the
$2\times2$ symbol
\[
\hat K(\alpha)=\begin{pmatrix}2&1+e^{-\mathrm i\alpha}\\1+e^{\mathrm i\alpha}&2\end{pmatrix},
\]
and $G_c$ the $4\times4$ symbol $8\Id+\hat K(\alpha)\otimes J+J\otimes\hat K(\beta)$.
In the basis $\frac1{\sqrt2}(1,\pm1)$ of eigenvectors of $J$, $J$ becomes
$Z=\mathrm{diag}(1,-1)$ and
\[
\hat K(\alpha)=2\Id+Z+A_\alpha,\qquad
A_\alpha=\begin{pmatrix}\cos\alpha&\mathrm i\sin\alpha\\-\mathrm i\sin\alpha&-\cos\alpha\end{pmatrix},
\qquad A_\alpha^2=\Id .
\]
Hence the symbol equals
\[
8\Id+2\,(\Id\otimes Z+Z\otimes\Id+Z\otimes Z)\;+\;(A_\alpha\otimes Z+Z\otimes A_\beta).
\]
The first part is diagonal with entries $8+2(s+t+st)$, $s,t=\pm1$, that is
$14,6,6,6$; the second has norm at most $2$, since $\|A_\alpha\|=\|Z\|=1$.  By
Weyl's inequality the eigenvalues of the symbol lie in $[6-2,14+2]=[4,16]$, and
so does $\operatorname{spec}(G_1)$.
\end{proof}
The measured extreme eigenvalues of $G_1$ are $4.22,4.07,4.02,4.005$ and
$15.66,15.90,15.97,15.99$ for $N=8,\dots,64$, so the enclosure is sharp.

Thus $G_1^{-1}$ is applied by a Chebyshev polynomial of the stencil on
$[4,16]$ with error $2\cdot3^{-d}$ after $d$ steps, each step $\cO(N^2)$ without
transforms; $d=14$ gives $3\cdot10^{-7}$, far below what the outer iteration
needs from a preconditioner.

Theorem 4.1 of \cite{LORR} writes $\Sone=U\Sigma V^\top$ with an explicit orthonormal
$V$ (sines on the $N/2\times N/2$ grid of $2\times2$ cells times four cell patterns),
but $U$ is not orthonormal (eq.\ (5.2) there), so $G_1=V\Sigma U^\top U\Sigma V^\top$ is
not diagonal in the basis $V$: in that basis the coupling between different
cell frequencies has entries up to $0.6$.  Their
Theorem 5.1 circumvents this by a truncation $M$ of the level-1 data after
which the remaining columns of $U$ become orthonormal (after pairing the
types $p=0,2$), giving an explicit left inverse $B^\dagger M$, $B:=M\Sone$.  The
truncation is the following.

\begin{lemma}[The truncated level-1 Gram of \cite{LORR}]\label{lem:lorr}
Let $M$ delete, in every quadrant, the level-1 entries $(h,s)$ with $h$ and $s$
both odd (one of the four parity classes).  Then
${\Sone}^\top M\Sone=4\Id+8\Pi$, where $\Pi$ is the orthogonal projection onto
images constant on the $2\times2$ cells.  Hence $B=M\Sone$ has singular values
$\sqrt{12}$ (on the cell-constant images) and $2$ (on their complement), and
$B^\dagger=\big(\tfrac14(\Id-\Pi)+\tfrac1{12}\Pi\big)\,{\Sone}^\top M$.
\end{lemma}
\begin{proof}
Direct computation of the four quadrant stencils with the odd--odd rows removed
(verified to machine precision for $N=8$; the other three parity classes do
not give a block-diagonal Gram).  Equivalently: in eqs.\ (5.4)--(5.6) of \cite{LORR}
the truncated vectors $\bar w^{(p)}$ contain only copies of $\phi_{(k_1,k_2)}$,
so their norms do not depend on $(k_1,k_2)$: $\|\bar w^{(0)}\pm\bar w^{(2)}\|^2/2=12,4$
and $\|\bar w^{(1)}\|^2=\|\bar w^{(3)}\|^2=4$.
\end{proof}

So the two orthogonal transformations of \cite[Thm.~5.1]{LORR} reduce to the
cell average $\Pi$; no sine transform is needed.  As a preconditioner for
$G_h$, $(4\Id+8\Pi)^{-1}$ is weaker than $G_1^{-1}$: at $\tau=0.2N$,
$\kappa\big((4\Id+8\Pi)^{-1}G_h\big)=4.25,5.45,7.09$ against $\kappa(G_1^{-1}G_h)=3.17,4.16,5.44$
and $\kappa(G_h)=4.48,5.64,7.69$ for $N=8,16,32$;
the truncation discards the rows that couple neighbouring cells, which is
where $G_h$ differs most from a multiple of the identity.  We therefore use
the exact $G_1^{-1}$ (Lemma~\ref{lem:g1}) in the fast implementation of
Section~\ref{sec:fast}.
\subsection{Conditioning and accuracy}\label{sec:hybnum}
All operators were formed densely ($N\le64$), with the Tikhonov filter
$p(\lambda)=(\lambda+\tau)^{-1}$ to which the Chebyshev filters agree to $\varepsilon/\tau$.
The left-inverse identity holds to $10^{-15}$--$10^{-13}$ for both constructions,
$L^n_\tau$ of Remark~\ref{rem:filter} and the hybrid.

\begin{table}[h]\centering\small
\begin{tabular}{lS[table-format=2.0]S[table-format=1.2]S[table-format=2.2]S[table-format=1.2]S[table-format=2.2]S[table-format=1.2]S[table-format=1.2]}
\toprule
$\tau$ & {$N$} & {$\kappa(E^\top FE)$} & {$\kappa(G_h)$} & {$\kappa(G_1)$} & {$\kappa(G_1^{-1}G_h)$} & {err$/$MP, original} & {err$/$MP, hybrid}\\
\midrule
$0.2N$ & 8  & 2.63 & 4.48 & 3.71 & 3.17 & 1.30 & 1.14\\
       & 16 & 3.15 & 5.64 & 3.91 & 4.16 & 1.28 & 1.15\\
       & 32 & 3.84 & 7.69 & 3.97 & 5.44 & 1.31 & 1.16\\
       & 64 & 4.64 & 10.22 & 3.99 & 7.21 & 1.35 & 1.19\\
\midrule
$N$    & 8  & 3.65 & 7.22 & 3.71 & 4.30 & 1.06 & 1.02\\
       & 16 & 4.99 & 11.07 & 3.91 & 6.62 & 1.06 & 1.03\\
       & 32 & 6.88 & 17.00 & 3.97 & 9.40 & 1.07 & 1.02\\
       & 64 & 9.51 & 25.06 & 3.99 & 13.32 & 1.08 & 1.03\\
\bottomrule
\end{tabular}
\caption{Conditioning of the outer operators and reconstruction error on
$1\%$ noise relative to the Moore--Penrose inverse (white-noise image).
$E^\top FE$: Section~\ref{sec:oblique}; $G_h$: hybrid without preconditioner; $G_1^{-1}G_h$:
hybrid with SPIFE's level-1 inverse as preconditioner.}
\label{tab:hyb}
\end{table}

\begin{table}[h]\centering\small
\begin{tabular}{l*{6}{S[table-format=3.2]}}
\toprule
$N$ & {8} & {16} & {32} & {64} & {128} & {256}\\
\midrule
median $\lambda_{\max}(M'_k)/\lambda_{\max}(M_k)$ & 0.40 & 0.44 & 0.46 & 0.46 & 0.46 & 0.46\\
$\sum_kd_k/N$, full Gram $M_k$ ($\tau=0.2N$) & 102 & 116 & 126 & 136 & 144 & 151\\
$\sum_kd_k/N$, upper levels $M'_k$ ($\tau=0.2N$) & 62 & 73 & 83 & 90 & 97 & 103\\
$\sum_kd_k/N$, full / upper ($\tau=N$) & {48/31} & {54/35} & {58/39} & {63/43} & {66/46} & {70/48}\\
\bottomrule
\end{tabular}
\caption{Inner degrees ($\varepsilon=10^{-10}$, exact $\lambda_{\max}$ as bounds).
$\lambda_{\max}(M'_k)/\lambda_{\max}(M_k)$ is $0.25$ at $\omega\to0$ and $0.5$ at large
$|\omega|$, the two singular values of the removed level being $2\cos\frac\omega4$
and $2\sin\frac\omega4$.}
\label{tab:deg}
\end{table}

\subsubsection*{Accuracy.}  Moving the filter to the upper levels brings the
inverse closer to Moore--Penrose at equal $\tau$: $1.14$--$1.19\times$ against
$1.30$--$1.35\times$ at $\tau=0.2N$, and $1.02$--$1.03\times$ against $1.06$--$1.08\times$
at $\tau=N$ (Table~\ref{tab:hyb}).  The level-1 anti-phase directions, which
the full-Gram filter damps by $\lambda/(\lambda+\tau)$ together with everything
else, are now inverted exactly by ${\Sone}^\dagger$; their singular value
$2\sin\frac\omega4$ is harmless once the four quadrants are fused, since the
level-1 Gram $G_1$ has $\kappa\le4$ for all $N$ (Table~\ref{tab:hyb}).

\subsubsection*{Cost.}  The upper-level Grams have $\lambda_{\max}(M'_k)\approx0.46\lambda_{\max}(M_k)$,
so the inner degree sums drop to $0.65$--$0.70$ of those of $L^n_p$
(Table~\ref{tab:deg}).  The outer conditioning is worse, $\kappa(G_1^{-1}G_h)/\kappa(E^\top FE)=1.2,1.3,1.4,1.55$
for $N=8,\dots,64$ at $\tau=0.2N$ ($1.2$--$1.4$ at $\tau=N$), so $D$ grows by
$\sqrt{1.2}$--$\sqrt{1.55}$.  The product $(\sum_kd_k)D$ is $0.83\times$ that of $L^n_p$ at $N=64$ for both $\tau$, plus one $\cO(N^2)$ application of
$G_1^{-1}$ per outer step.  Without the preconditioner the hybrid would be
worse: $\kappa(G_h)=4.5$--$10.2$ ($\tau=0.2N$) and $7$--$25$ ($\tau=N$), so SPIFE's
explicit ${\Sone}^\dagger$ is what makes the arrangement competitive; it halves
the outer conditioning.
\subsection{Fast implementation and certified bounds}\label{sec:fast}
The implementation realizes \eqref{eq:hyb} with the butterflies of
Section~\ref{sec:fourier} restricted to levels $2,\dots,n$, the level-1 butterfly
$\mathcal B^n_1$ at both ends, $G_1^{-1}$ by the Chebyshev polynomial of
Lemma~\ref{lem:g1}, and a fixed-degree outer Chebyshev iteration on
$G_1^{-1}G_h$ whose interval comes from a few Lanczos steps on the pencil
$(G_h,G_1)$ with the residual enclosure, as in Section~\ref{sec:bounds}.  The bounds
$\Lambda'_k\ge\lambda_{\max}(M'_k)$ are either $1.5\times$ a power iteration or
certified: the row-sum recursion \eqref{eq:cert} started at level $2$
with the identity as base (no
violation against the true row sums for $N\le512$, and $\Lambda'_k\ge\lambda_{\max}(M'_k)$
at every sampled frequency up to $N=1024$).  The certified degree sum of the
hybrid is $0.58$--$0.59$ times that of $L^n_p$ at every $N$ from $8$ to
$1024$, and exceeds the power-iteration one by the factors $1.0,1.15,1.3,1.5,1.7,1.9$ for
$N=32,\dots,1024$, the same growth as for the full Gram.

Costs are counted as butterfly flops: the number of (level, offset-frequency
row) pairs processed, times $N$; FFTs and the $\cO(N^2)$ stencils, which are
negligible, are left out, and $L^n_p$ is counted the same way.
Divided by $N^2\log_2N$ the count grows like the total degree
$(\sum_kd_k)D$, i.e.\ like $N^{1/2}/\sigma_{\min}(\nR)$ up to logarithms.  Table~\ref{tab:fasthyb} in Section~\ref{sec:numhyb} compares the two inverses.
\section{Numerical results}\label{sec:num}
\subsection{The polynomial-filter inverse}
All wall times reported below are for the compiled implementation of the
row kernels described in Section~\ref{sec:compiled} (C\texttt{++}, one core of an
otherwise idle node of the compute2 cluster);
the operators themselves were verified in a NumPy prototype, with a dense
version and the ADRT itself checked against the \texttt{adrt} package
\cite{Otness23}, whose timings depend too strongly on the implementation to be
meaningful, and the C\texttt{++} kernels reproduce it to $10^{-15}$.
Appendix~\ref{app:code} lists the script behind each result.  Table~\ref{tab:fast} lists, for $\tau=0.2N$ and
$\varepsilon=10^{-10}$, the inner degree sums, the outer degree and conditioning,
the number of frequency-steps per application (each an $\cO(N\log N)$ butterfly
pair), the left-inverse error $\|L^n_p\nR[x]-x\|/\|x\|$ on a random image, and
the reconstruction error on data with $1\%$ white noise relative to
$(\nR)^\dagger$ (computed by CGLS \cite{Saad} to $10^{-10}$).

\begin{table}[h]\centering\small
\begin{tabular}{l*{6}{S[table-format=3.2]}}
\toprule
$N$ & {8} & {16} & {32} & {64} & {128} & {256}\\
\midrule
$\sum_kd_k/N$ & 124 & 141 & 154 & 165 & 175 & 184\\
outer degree $D$ & 21 & 24 & 26 & 29 & 32 & 37\\
$\kappa(E^\top F_pE)$ & 3.2 & 3.8 & 4.7 & 5.7 & 7.0 & 9.2\\
frequency-steps per application $/(10^4N)$ & 1.04 & 1.35 & 1.60 & 1.92 & 2.24 & 2.72\\
left-inverse error $/10^{-12}$ & 5 & 3 & 8 & 8 & 10 & 10\\
error ratio to $(\nR)^\dagger$ & 1.29 & 1.31 & 1.35 & 1.36 & 1.36 & 1.36\\
\bottomrule
\end{tabular}
\caption{The polynomial-filter inverse, $\tau=0.2N$, $\varepsilon=10^{-10}$.  The
degree sum grows like $N\sqrt{\log N}$ and the outer degree like
$\sqrt{\kappa(E^\top F_pE)}$ (in this range of $N$ compatible with $\sqrt{\log N}$ and
with $N^{0.2}$, see Section~\ref{sec:cond}); the
noise amplification saturates at $1.36\times$ Moore--Penrose.  With $\tau=N$ the
count drops to $\approx10^4N$ and the ratio to $1.06$ ($N\le64$).}
\label{tab:fast}
\end{table}

The left-inverse error ($\approx10^{-11}$) is the polynomial consistency error
of the Chebyshev filters at $\varepsilon=10^{-10}$ and decreases with $\varepsilon$; with
exact per-frequency factors it is $10^{-15}$.
\subsection{Comparison of the setup choices}
Table~\ref{tab:setup} and Figure~\ref{fig:setup} compare the choices for the
data-independent setup: the spectral bounds $\Lambda_k$ by power iteration
(safety factor $1.5$) or by the explicit formula \eqref{eq:Lambda}, and the
outer interval by the Weinstein enclosure, by the computation-free choices
$(2/\log N,4.2)$, $(4/\log N,4.2)$, $(0.5,4.2)$, and, as a control, by the
overestimate $(8/\log N,4.2)$ which exceeds $\lambda_{\min}(E^\top F_pE)$ at every
size.  All safe choices give an exact left inverse ($10^{-12}$--$10^{-11}$) and
identical reconstructions (error ratio $1.31$--$1.36$ at $\tau=0.2N$, $1.07$ at
$\tau=N$); they differ only in work.  The explicit $\Lambda_k$ costs $12$--$24\%$ more
inner steps than the power-iteration bounds; the computation-free outer
interval costs a longer outer polynomial, from $+20\%$ ($4/\log N$, whose margin
shrinks with $N$) to $+90\%$ ($2/\log N$).  The overestimate is the one choice
that fails: with the interval wrong, the degree-$19$--$23$ polynomial leaves an
error of $3$--$5\cdot10^{-6}$ in the identity $L^n_p\nR=\Id$, three times the
tolerance in the exponent, for all $N$.  At $\tau=N$ the inner degrees halve
and the total work is the smallest of all, $1.0$--$1.8\cdot10^4N$ steps, with the
reconstruction closest to Moore--Penrose; its computation-free interval
$(1/\log N,4.2)$ (the smaller $\lambda_{\min}$ at $\tau=N$ calls for the smaller
constant) again costs $50$--$75\%$ more outer work.

\begin{table}[h]\centering\small
\begin{tabular}{llS[table-format=3.0]S[table-format=2.0]S[table-format=1.2e1]S[table-format=1.1e-2]S[table-format=2.2]@{\ s\hspace{\tabcolsep}}}
\toprule
$N$ & setup ($\Lambda_k$ / outer interval) & {$\sum d_k/N$} & {$D$} & {steps$/N$} & {$\|L^n_p\nR x-x\|/\|x\|$} & \multicolumn{1}{c}{setup}\\
\midrule
32 & power / Weinstein & 154 & 24 & 1.48e4 & 2.6e-12 & 0.69\\
   & explicit / Weinstein & 173 & 24 & 1.66e4 & 2.6e-12 & 0.60\\
   & explicit / $(2/\log N,4.2)$ & 173 & 40 & 2.77e4 & 1.2e-11 & 0\\
   & explicit / $(4/\log N,4.2)$ & 173 & 28 & 1.94e4 & 5.9e-12 & 0\\
   & explicit / $(0.5,4.2)$ & 173 & 35 & 2.42e4 & 1.6e-11 & 0\\
   & explicit / $(8/\log N,4.2)$, overestimate & 173 & 19 & 1.31e4 & 5.1e-6 & 0\\
   & explicit / Weinstein, $\tau=N$ & 79 & 32 & 1.01e4 & 9.9e-12 & 0.37\\
   & explicit / $(1/\log N,4.2)$, $\tau=N$ & 79 & 56 & 1.77e4 & 2.3e-11 & 0\\
\midrule
64 & power / Weinstein & 165 & 26 & 1.72e4 & 6.4e-12 & 3.4\\
   & explicit / Weinstein & 196 & 26 & 2.04e4 & 6.4e-12 & 3.7\\
   & explicit / $(2/\log N,4.2)$ & 196 & 43 & 3.37e4 & 2.2e-11 & 0\\
   & explicit / $(4/\log N,4.2)$ & 196 & 31 & 2.43e4 & 6.3e-12 & 0\\
   & explicit / $(0.5,4.2)$ & 196 & 35 & 2.74e4 & 1.6e-11 & 0\\
   & explicit / $(8/\log N,4.2)$, overestimate & 196 & 21 & 1.65e4 & 2.8e-6 & 0\\
   & explicit / Weinstein, $\tau=N$ & 89 & 38 & 1.36e4 & 1.1e-11 & 1.9\\
   & explicit / $(1/\log N,4.2)$, $\tau=N$ & 89 & 61 & 2.18e4 & 2.9e-11 & 0\\
\midrule
128 & power / Weinstein & 175 & 29 & 2.03e4 & 8.3e-12 & 22\\
   & explicit / Weinstein & 217 & 29 & 2.52e4 & 8.2e-12 & 22\\
   & explicit / $(2/\log N,4.2)$ & 217 & 47 & 4.08e4 & 1.8e-11 & 0\\
   & explicit / $(4/\log N,4.2)$ & 217 & 33 & 2.87e4 & 1.1e-11 & 0\\
   & explicit / $(0.5,4.2)$ & 217 & 35 & 3.04e4 & 1.6e-11 & 0\\
   & explicit / $(8/\log N,4.2)$, overestimate & 217 & 23 & 2.00e4 & 2.6e-6 & 0\\
   & explicit / Weinstein, $\tau=N$ & 99 & 45 & 1.77e4 & 1.3e-11 & 10\\
   & explicit / $(1/\log N,4.2)$, $\tau=N$ & 99 & 66 & 2.60e4 & 2.9e-11 & 0\\
\bottomrule
\end{tabular}
\caption{Setup choices for $L^n_p$ ($\varepsilon=10^{-10}$; $\tau=0.2N$ unless stated).
``steps'' are butterfly frequency-steps per application, each $\cO(N\log N)$;
``setup'' is the wall-clock of the data-independent setup on one core with
the C\texttt{++} row kernels (the Weinstein enclosure needs a few Lanczos steps, which
dominate, the power iteration $40$ butterfly products per frequency).  The error ratio to the
Moore--Penrose inverse on $1\%$ noise is $1.31,1.33,1.36$ ($\tau=0.2N$) and $1.07$
($\tau=N$) for every row.}
\label{tab:setup}
\end{table}

\begin{figure}[t]\centering
\includegraphics[width=0.98\textwidth]{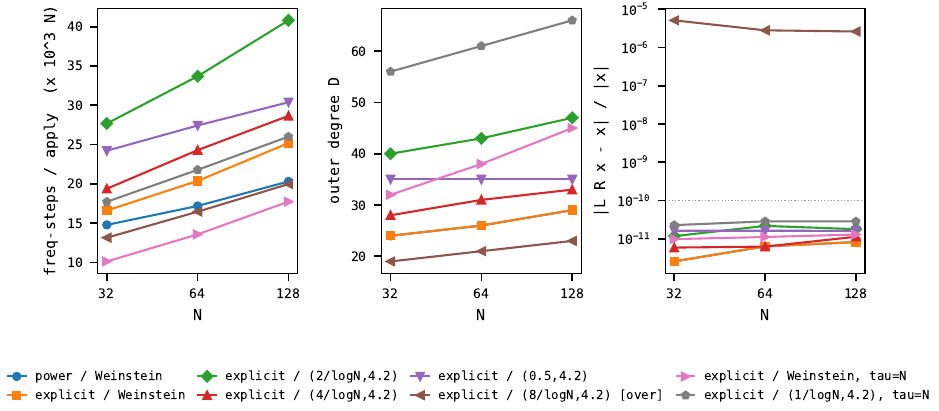}
\caption{The setup choices of Table~\ref{tab:setup}: work per application
(left), outer degree (middle) and left-inverse error (right) against $N$.  All
safe choices sit at the $10^{-11}$ level set by the polynomial consistency;
the overestimated interval (brown) fails by five orders of magnitude.}
\label{fig:setup}
\end{figure}
\subsection{The fast hybrid}\label{sec:numhyb}
Table~\ref{tab:fasthyb} compares $L^n_p$ with the
hybrid of Section~\ref{sec:fast} on a white-noise image ($\tau=0.2N$,
$\varepsilon=10^{-10}$, power-iteration bounds); wall times are for the C\texttt{++} row
kernels on one core.

\begin{table}[h]\centering\footnotesize\setlength{\tabcolsep}{4pt}
\begin{tabular}{llS[table-format=3.0]S[table-format=1.3]S[table-format=1.2]S[table-format=2.0]S[table-format=2.0]S[table-format=2.2]S[table-format=1.2]}
\toprule
$N$ & inverse & {$\sum_kd_k/N$} & \multicolumn{2}{c}{outer interval} & {$D$} & {$\frac{\text{flops}}{10^3N^2\log_2N}$} & {apply [s]} & {$\frac{\text{err}}{\text{MP}}$, $1\%$}\\
\cmidrule(lr){4-5}
 & & & {from} & {to} & & & & \\
\midrule
64  & $L^n_p$ & 165 & 0.85 & 3.95 & 26 & 34 & 0.94 & 1.35\\
    & hybrid, $G_1^{-1}$ & 110 & 0.125 & 0.92 & 33 & 24 & 0.68 & 1.19\\
    & hybrid, $(4\Id+8\Pi)^{-1}$ & 110 & 0.150 & 1.43 & 38 & 28 & 0.76 & 1.19\\
\midrule
128 & $L^n_p$ & 175 & 0.70 & 3.97 & 29 & 41 & 5.0 & 1.37\\
    & hybrid, $G_1^{-1}$ & 118 & 0.098 & 0.94 & 38 & 31 & 3.9 & 1.20\\
    & hybrid, $(4\Id+8\Pi)^{-1}$ & 118 & 0.119 & 1.44 & 43 & 35 & 4.3 & 1.20\\
\midrule
256 & $L^n_p$ & 184 & 0.53 & 3.98 & 34 & 50 & 27.6 & 1.36\\
    & hybrid, $G_1^{-1}$ & 125 & 0.080 & 0.95 & 42 & 37 & 20.7 & 1.20\\
\bottomrule
\end{tabular}
\caption{Fast implementations at $\tau=0.2N$, $\varepsilon=10^{-10}$.
Left-inverse identity $10^{-11}$ in all cases.  Wall times: C\texttt{++} row kernels,
one core.}
\label{tab:fasthyb}
\end{table}

The inner degree sums are $0.67\times$ those of $L^n_p$, the outer degrees
$1.3\times$, and the cost per application $0.70$--$0.76\times$; the wall times follow
($0.72$, $0.78$, $0.75\times$ for $N=64,128,256$).
The explicit truncated inverse of Lemma~\ref{lem:lorr} is cheaper per step
but needs more outer steps and loses overall, by $10$--$12\%$.

\subsection{The inner tolerance}\label{sec:inner}
Theorem~\ref{thm:fast}(i) says that $L^n_p\nR=\Id$ for every choice of the
inner degrees $d_k$; the tolerance $\varepsilon_{\rm in}$ to which $p_k$
approximates $(\lambda+\tau)^{-1}$ only shapes the oblique projection.  Table~\ref{tab:inner}
varies $\varepsilon_{\rm in}$ with the outer polynomial kept at $10^{-10}$.  Nothing
measurable changes except the cost: the left-inverse identity stays at
$10^{-11}$, the outer interval moves in the third digit, and the noise
amplification by at most $0.007$, while the inner degree sums fall from
$165$--$184N$ to $22$--$24N$ ($L^n_p$) and from $110$--$125N$ to $16$--$17N$ (hybrid),
and the wall times by a factor $5$--$7$.  The degree $d_k=\lceil\tfrac12\sqrt{1+\Lambda_k/\tau}\ln(2/\varepsilon_{\rm in})\rceil$
at $\varepsilon_{\rm in}=0.1$ is $\approx1.5\sqrt{1+\Lambda_k/\tau}$, essentially the
minimal degree that resolves the pole down to the scale $\tau$; the
projector is insensitive to the shape of the filter beyond that.  This is the
setting we recommend (with $\tau$ between $N$ and $5N$ for data that do not come
from the ADRT, Section~\ref{sec:skimage}); the rigorous bound $\lambda_{\max}(E^\top F_pE)\le4(1+\varepsilon_{\rm in})$
of Section~\ref{sec:bounds} becomes $4.4$ (measured $4.06$).

\begin{table}[h]\centering\small
\begin{tabular}{llS[table-format=3.0]S[table-format=3.0]S[table-format=2.2]S[table-format=2.2]S[table-format=1.3]S[table-format=1.3]}
\toprule
 & & \multicolumn{2}{c}{$\sum_kd_k/N$} & \multicolumn{2}{c}{apply [s]} & \multicolumn{2}{c}{err$/$MP}\\
\cmidrule(lr){3-4}\cmidrule(lr){5-6}\cmidrule(lr){7-8}
$N$ & $\varepsilon_{\rm in}$ & {$L^n_p$} & {hybrid} & {$L^n_p$} & {hybrid} & {$L^n_p$} & {hybrid}\\
\midrule
64  & $10^{-10}$ & 165 & 110 & 0.94 & 0.68 & 1.353 & 1.190\\
    & $10^{-2}$ & 38 & 25 & 0.23 & 0.20 & 1.354 & 1.191\\
    & $10^{-1}$ & 22 & 15 & 0.14 & 0.14 & 1.360 & 1.194\\
\midrule
128 & $10^{-10}$ & 175 & 118 & 5.0 & 3.9 & 1.367 & 1.199\\
    & $10^{-2}$ & 40 & 27 & 1.20 & 1.04 & 1.368 & 1.200\\
    & $10^{-1}$ & 23 & 16 & 0.75 & 0.71 & 1.373 & 1.205\\
\midrule
256 & $10^{-10}$ & 184 & 125 & 27.6 & 20.8 & 1.360 & 1.204\\
    & $10^{-2}$ & 42 & 29 & 6.7 & 5.7 & 1.361 & 1.205\\
    & $10^{-1}$ & 24 & 17 & 4.0 & 3.6 & 1.364 & 1.207\\
\bottomrule
\end{tabular}
\caption{Inner tolerance $\varepsilon_{\rm in}$ at fixed outer tolerance $10^{-10}$,
$\tau=0.2N$, white-noise image with $1\%$ noise; left-inverse identity
$10^{-12}$--$10^{-11}$ in every row.  Wall times: C\texttt{++} row kernels, one core.}
\label{tab:inner}
\end{table}

\subsection{Compiled implementation and comparison with the normal equations}\label{sec:compiled}
Wall times depend strongly on the implementation (in a NumPy prototype the
ADRT itself runs at about $0.2$\,Gflop/s against $1$\,Gflop/s for the
butterflies), so all timings in this paper are for compiled code.  The row
kernels of the oblique inverses (the butterflies, their adjoints and the
Chebyshev filter at one offset frequency, in cache) are written in C\texttt{++}, keeping the quadrant packing,
the FFTs, the dense low-frequency table and the outer polynomial of
Section~\ref{sec:inner} ($\varepsilon_{\rm in}=0.1$), and the normal-equation
polynomial of Remark~\ref{rem:normal} was run with the compiled ADRT and
back-projection of the \texttt{adrt} package \cite{Otness23} (version 1.2.0 as
distributed).  Table~\ref{tab:compiled} reports one application on one core of
an otherwise idle node; the left-inverse identity holds to $10^{-11}$ for the
oblique inverses and the normal-equation polynomial reaches the least-squares
solution to $7\cdot10^{-11}$.

\begin{table}[h]\centering\small
\begin{tabular}{l*{4}{S[table-format=2.3]}}
\toprule
$N$ & {normal equations [s]} & {$L^n_p$ [s]} & {hybrid [s]} & {hybrid, dense table [s]}\\
\midrule
64  & 0.064 & 0.140 & 0.123 & 0.097\\
128 & 0.552 & 0.750 & 0.665 & 0.517\\
256 & 4.93  & 4.01  & 3.27  & 2.80\\
512 & 56.1  & 22.4  & 18.9  & 16.9\\
\midrule
outer degree, $N=64,\dots,512$ & \multicolumn{1}{c}{198--1121} & \multicolumn{1}{c}{26--40} & \multicolumn{1}{c}{34--46} & \multicolumn{1}{c}{34--46}\\
error$/$least squares, $1\%$ & \multicolumn{1}{c}{1.00} & \multicolumn{1}{c}{1.36--1.38} & \multicolumn{1}{c}{1.21} & \multicolumn{1}{c}{1.21}\\
\bottomrule
\end{tabular}
\caption{One application on one core: the normal-equation Chebyshev polynomial
with the compiled ADRT of the \texttt{adrt} package, and the oblique inverses
with C\texttt{++} row kernels ($\tau=0.2N$, $\varepsilon_{\rm in}=0.1$, outer tolerance
$10^{-10}$; hybrid with $G_1^{-1}$ of degree $6$, dense table for inner degrees
above $8$).  Compute2 cluster, exclusive node.}
\label{tab:compiled}
\end{table}

The normal-equation polynomial is the fastest method up to $N=128$; from
$N=256$ on the oblique inverses are faster, $L^n_p$ by $1.2$ and $2.5$ times
and the hybrid with the dense table by $1.8$ and $3.3$ times at $N=256$ and
$512$.  The gain from $256$ to $512$ exceeds the factor $\sqrt2$ of the
operation counts because the normal-equation time grows by $11.4$ against the
$8.5$ predicted by its degree and transform cost, presumably through the
memory traffic of the full transform, whereas the oblique inverses work on
one offset-frequency row of $N$ entries at a time.  That this is not an
intrinsic advantage of the row-wise organization was checked directly:
evaluating ${\nR}^\top\nR$ itself row by row in the offset-Fourier domain with a C\texttt{++}
kernel of the same kind reproduces the transform pair to $10^{-15}$ but is
twice as slow as the compiled pair, the FFTs and the complex arithmetic
outweighing the locality.  The hybrid is the faster of the two oblique inverses,
as in Tables~\ref{tab:fasthyb} and~\ref{tab:inner}: its inner
work, which the C\texttt{++} kernels accelerate most, is two thirds of that of
$L^n_p$, and its dense table holds fewer than half as many rows.  All three methods
parallelize over offset frequencies or within each transform; with four
threads the C\texttt{++} oblique inverses run $2.2$ to $3$ times faster, but the
distributed \texttt{adrt} package is single-threaded, so the table compares
one core.

\subsection{A Shepp--Logan example}
Figure~\ref{fig:shepp} shows the textbook test: the (modified) Shepp--Logan
phantom on a $128\times128$ grid, data $b=\nR[x]+\nu$ with white noise
$\nu$ of relative size $\|\nu\|/\|\nR[x]\|=0,1\%,5\%,10\%$, and
the reconstructions by the explicit oblique inverse $L^n_p$ ($\tau=0.2N$,
$\varepsilon=10^{-10}$), by the hybrid of Section~\ref{sec:hyb} with the same
parameters, and by CGLS run to relative residual $10^{-8}$ (85--92 iterations;
this is the standard inversion and, since $\kappa(\nR)\approx10$ at this size, it
is the least-squares solution rather than a semi-converged one);
Table~\ref{tab:shepp} adds $L^n_p$ at $\tau=N$.  On noise-free data the oblique
inverses recover the phantom to $10^{-11}$ (CGLS: $8\cdot10^{-7}$ at its
stopping tolerance).  With noise, the relative errors are exactly proportional
to the noise level for all methods, as the inverses are linear, with the
constant ratios $1.38$ ($L^n_p$, $\tau=0.2N$), $1.21$ (hybrid, $\tau=0.2N$) and
$1.08$ ($L^n_p$, $\tau=N$) to the least-squares error; visually the
reconstructions are indistinguishable at $1\%$ and differ only in the grain of
the noise at $5$--$10\%$.  The cost of $L^n_p$ here is $2.2\cdot10^4N$ (resp.\
$1.1\cdot10^4N$ at $\tau=N$) butterfly frequency-steps per reconstruction and
that of the hybrid $0.76$ times as many flops (Table~\ref{tab:fasthyb}),
independent of the data, versus $85$--$91$ CGLS iterations of two full ADRTs
each.

\begin{table}[h]\centering
\begin{tabular}{l*{4}{S[table-format=1.2e-2]}}
\toprule
relative noise & {$0$} & {$1\%$} & {$5\%$} & {$10\%$}\\
\midrule
$L^n_p$, $\tau=0.2N$ & 1.2e-11 & 3.62e-2 & 1.81e-1 & 3.62e-1\\
hybrid, $\tau=0.2N$ & 1.1e-11 & 3.19e-2 & 1.60e-1 & 3.19e-1\\
$L^n_p$, $\tau=N$ & 2.3e-11 & 2.83e-2 & 1.42e-1 & 2.83e-1\\
CGLS ($10^{-8}$) & 8.1e-7 & 2.63e-2 & 1.32e-1 & 2.63e-1\\
\bottomrule
\end{tabular}
\caption{Relative reconstruction errors $\|\hat x-x\|/\|x\|$ for the Shepp--Logan
phantom, $N=128$.}
\label{tab:shepp}
\end{table}

\begin{figure}[p]\centering
\includegraphics[width=\textwidth]{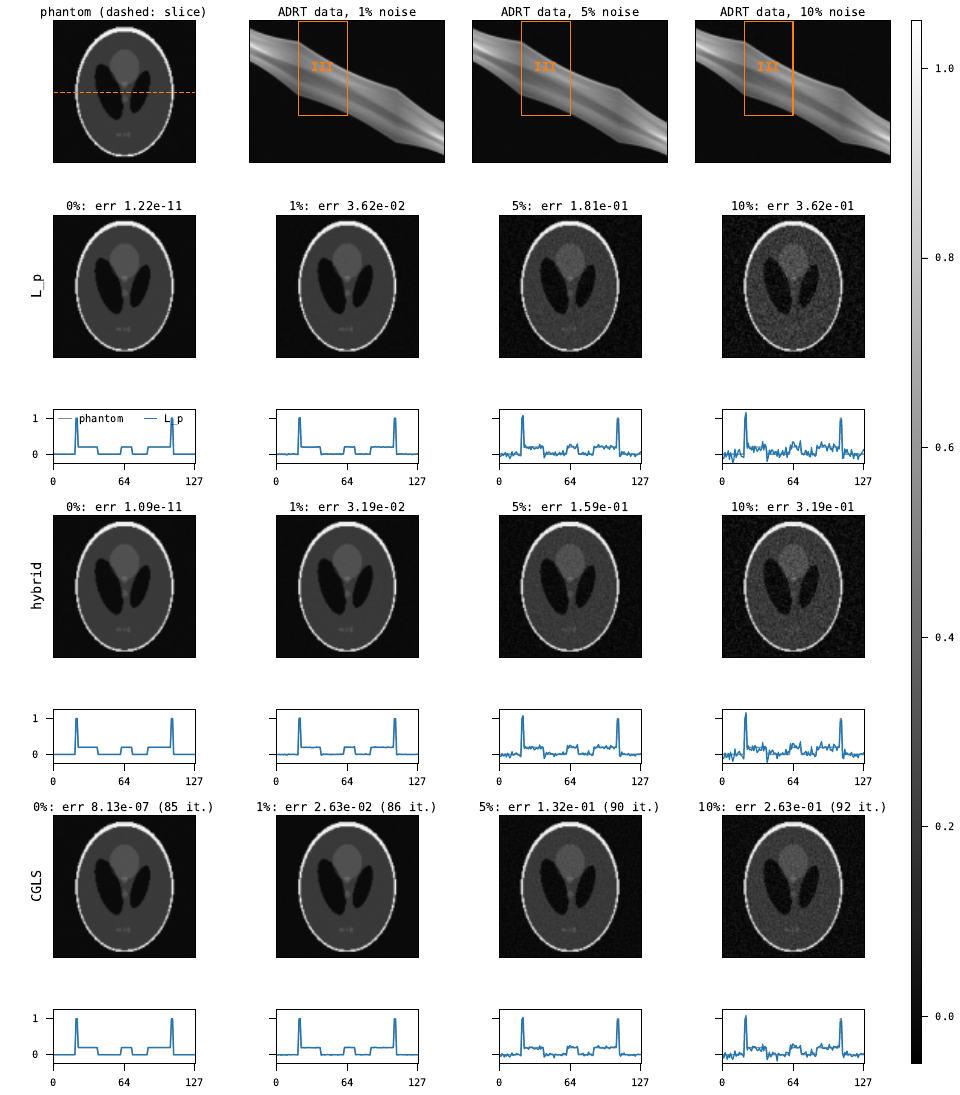}
\caption{Shepp--Logan phantom, $N=128$, $\tau=0.2N$.  Top row: the phantom
(the dashed line marks the slice shown below) and the full ADRT data at
$1\%,5\%,10\%$ relative noise, stitched across the four quadrants as in
\cite{Otness23}, with quadrant \RN{3} outlined.  Below, for the explicit
oblique inverse $L^n_p$, the hybrid and CGLS to relative residual $10^{-8}$:
the reconstructions at $0,1,5,10\%$ noise, each with the middle row of the
reconstruction against that of the phantom (grey).  Images and data share one
gray scale; the data are divided by the maximum of the noise-free data.
Errors as in Table~\ref{tab:shepp}.}
\label{fig:shepp}
\end{figure}

\subsection{Data from another discretization}\label{sec:skimage}
All data so far were generated by the ADRT itself.  Here they come from a
different discretization of the Radon transform: the straight-line projections
of \texttt{skimage.transform.radon} \cite{skimage}, which integrates along
lines by interpolation, converted to ADRT data in two ways.  The \emph{package
route} follows the tomography example of \cite{Otness23}: the projections are
computed at the angles of the ADRT lines and interpolated by a bicubic spline
at the continuous coordinates of every ADRT line; we multiply by $\cos\phi$, $\phi$
the angle of the line from the summation axis, to convert the arc-length
integral into the ADRT's one sample per column.  The \emph{direct route}
evaluates the projection on the straight line through the end points of each
digital line, by linear interpolation along the detector, with the sign and
centre conventions calibrated on a smooth test image.  The data differ from
$\nR x$ by the model mismatch $\|b-\nR x\|/\|\nR x\|$ of Table~\ref{tab:skimage}: a few
per cent, from the deviation of the digital lines from straight ones (up to
$\frac12\log N$ pixels) and from the interpolation.  We compare $L^n_p$ and the
hybrid ($\varepsilon_{\rm in}=0.1$) with the least-squares solution for the ADRT model
(CGLS to $10^{-10}$), for the smooth phantom of three Gaussians and the
Shepp--Logan phantom of Section~\ref{sec:num}, $N=128$.

\begin{table}[h]\centering\small
\begin{tabular}{llS[table-format=1.1]S[table-format=2.2]S[table-format=2.2]S[table-format=1.2]S[table-format=1.2]S[table-format=1.2]}
\toprule
 & & {mismatch} & {least squares} & \multicolumn{2}{c}{$L^n_p$$/$least squares} & \multicolumn{2}{c}{hybrid$/$least squares}\\
\cmidrule(lr){5-6}\cmidrule(lr){7-8}
route & phantom & {[\%]} & {error [\%]} & {$\tau=0.2N$} & {$\tau=5N$} & {$\tau=0.2N$} & {$\tau=5N$}\\
\midrule
package & smooth       & 2.6 & 4.26  & 1.20  & 0.99 & 1.05 & 1.00\\
        & Shepp--Logan & 8.2 & 39.6  & 1.15  & 1.00 & 1.08 & 1.00\\
direct  & smooth       & 0.7 & 0.37  & 10.4  & 1.33 & 6.6  & 1.16\\
        & Shepp--Logan & 5.0 & 25.9  & 1.46  & 1.03 & 1.34 & 1.02\\
\bottomrule
\end{tabular}
\caption{Data from \texttt{skimage.transform.radon}, no added noise, $N=128$:
model mismatch, error of the least-squares solution for the ADRT model, and
errors of the oblique inverses relative to it.}
\label{tab:skimage}
\end{table}

Two observations follow from Table~\ref{tab:skimage} and Figure~\ref{fig:skimage}.
First, the error is now dominated by the model: even the least-squares
solution, the best any inverse can do with the ADRT as the model, is at
$26$--$40\%$ for the Shepp--Logan phantom, because a mismatch of a few per cent
concentrated at the edges is amplified by the conditioning of the ADRT.  This
is a property of the ADRT as a model for straight-line data, not of the
inverses.  Second, the model error is not white, and the oblique inverses
treat it differently from least squares.  Least squares projects the
out-of-range part of the data orthogonally and discards it; an oblique
inverse projects along $\ker L$, so that a smooth, low-frequency mismatch can be
mapped into the image.  At $\tau=0.2N$ this costs up to a factor $10$ against least
squares (direct route, smooth phantom, where the mismatch is small and smooth),
and a factor $1.1$--$1.5$ when the mismatch is rougher.  With $1\%$ added noise the
ratios return to $1.2$--$1.9$.  Since $L^n_\tau$ tends to $(\nR)^\dagger$ as $\tau\to\infty$
(Remark~\ref{rem:filter}), a larger $\tau$ removes the effect: at $\tau=5N$ the
oblique inverses are within $1.00$--$1.33$ of least squares in every case, and at
$\tau=25N$ within $0.99$--$1.03$.  The cost hardly changes, as Section~\ref{sec:cond}
predicts for the total degree: on one core, $L^n_p$ takes $0.63$\,s at $\tau=5N$
against $0.76$\,s at $\tau=0.2N$ for $N=128$, and $3.8$\,s against $4.0$\,s for $N=256$
(the hybrid $0.91$ and $4.9$\,s against $0.75$ and $3.5$\,s); only at $\tau=25N$, where
the inner degrees reach their minimum, does the cost grow, by $1.4$--$2.7$ times.  For data that do not come from the ADRT
we therefore recommend $\tau$ between $N$ and $5N$.

\begin{figure}[p]\centering
\includegraphics[width=\textwidth]{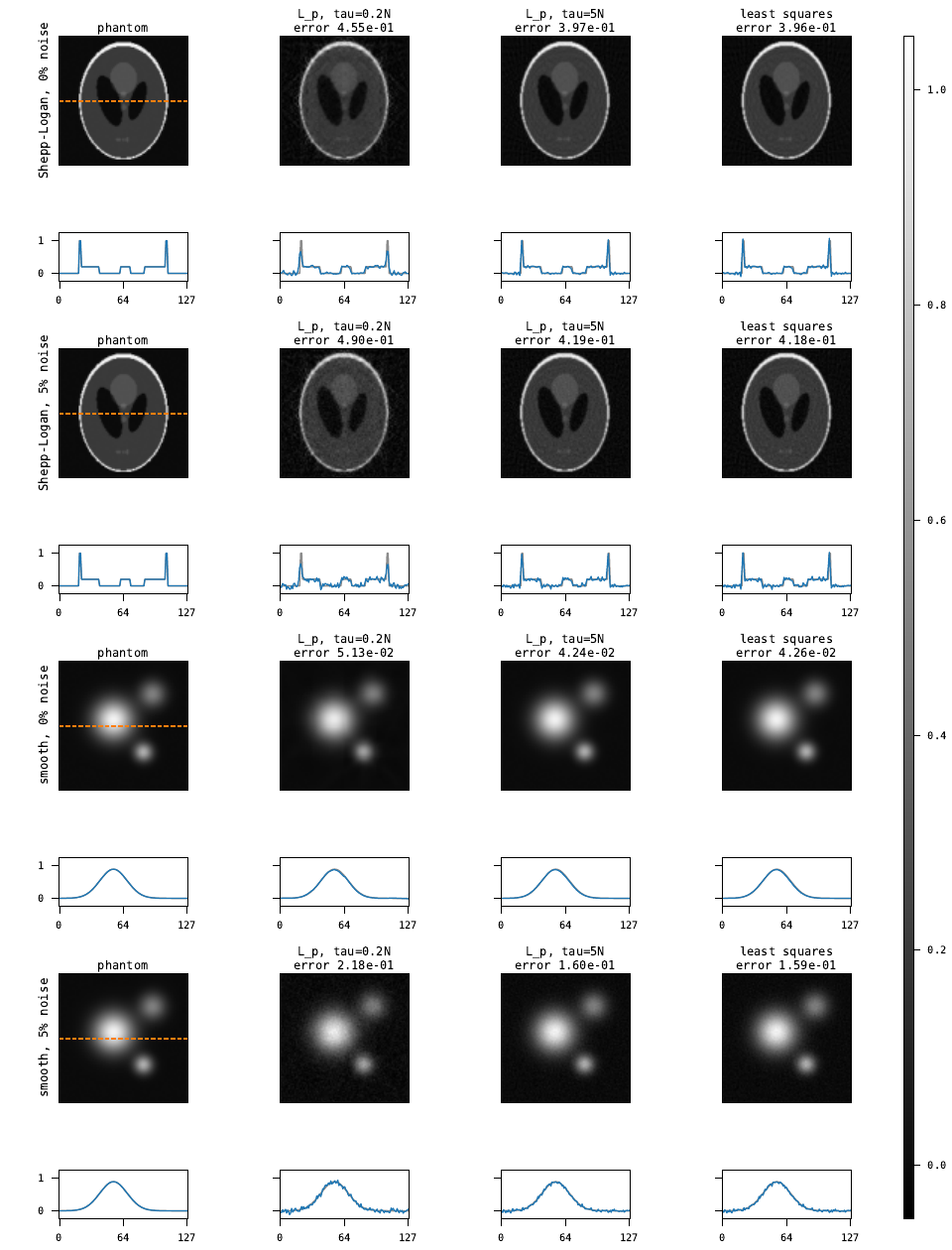}
\caption{Reconstructions from \texttt{skimage.transform.radon} data converted
by the package route, $N=128$, without added noise and with $5\%$ white noise
(rows): the phantom (dashed line: the slice shown below), $L^n_p$ at $\tau=0.2N$ and
$\tau=5N$, and the least-squares solution for the ADRT model.
Relative errors in the titles; slices of the middle row against the phantom
(grey) underneath.}
\label{fig:skimage}
\end{figure}

\begin{remark}[A Chebyshev polynomial for the Moore--Penrose inverse]\label{rem:normal}
A fixed-degree Chebyshev semi-iteration applied to the normal equations
${\nR}^\top\nR x={\nR}^\top b$ is also a data-independent polynomial,
\[
x_D=c_D({\nR}^\top\nR)\,{\nR}^\top b ,
\]
and in the same formal sense an ``explicit'' inverse.  Its interval
$[\lambda_{\min},\lambda_{\max}]({\nR}^\top\nR)$ is obtained exactly as for $L^n_p$,
by a few Lanczos steps with the Weinstein enclosure; $\lambda_{\max}({\nR}^\top\nR)=3.404N^2$
for all $N\le512$.  It differs from $L^n_p$ in three respects.
\begin{enumerate}
\item[(i)] \emph{Exactness.}  It is a left inverse only to the accuracy of the polynomial,
\[
\big\|c_D({\nR}^\top\nR){\nR}^\top\nR-\Id\big\|\le2\Big(\frac{\sqrt\kappa-1}{\sqrt\kappa+1}\Big)^D,
\]
whereas $L^n_p\nR=\Id$ holds for every choice of the inner degrees, which
affect only the oblique projection.  In practice both are limited by their
outer polynomial: at $\varepsilon=10^{-10}$ and $N=64,128,256$ the left-inverse
errors are $6\cdot10^{-11},8\cdot10^{-11},1.3\cdot10^{-9}$ for the normal
equations and $6\cdot10^{-12},8\cdot10^{-12},7\cdot10^{-12}$ for $L^n_p$.
\item[(ii)] \emph{Limit.}  Its limit is the Moore--Penrose inverse, not an oblique
projection: on $1\%$ noise its error is $1.000$ times that of $(\nR)^\dagger$
for $N=64$--$256$.
\item[(iii)] \emph{Cost.}  Since $\kappa({\nR}^\top\nR)=3.4N^2/\lambda_{\min}({\nR}^\top\nR)=0.54$--$0.62\,N^{3/2}$
(Section~\ref{sec:cond}), the degree for $10^{-10}$ is $D=44,71,117,198,340,595$
for $N=8,\dots,256$, that is $D\approx9N^{3/4}=\Theta(N/\sigma_{\min}(\nR))$, each
step costing one ADRT and one transpose.  The work is therefore
\[
\cO\Big(\frac{N^3\log N\log\varepsilon^{-1}}{\sigma_{\min}(\nR)}\Big)
\quad\text{against}\quad
\cO\Big(\frac{N^{5/2}\log^{3/2}N\log\varepsilon^{-1}}{\sigma_{\min}(\nR)}\Big)
\ \text{for } L^n_p .
\]
The oblique construction is asymptotically cheaper by $\sqrt N/\sqrt{\log N}$,
because its polynomials resolve the pole frequency by frequency instead of
through the whole transform; the factor $\sigma_{\min}(\nR)^{-1}$ is common to both.
\end{enumerate}
The constants favour the normal equations at small $N$, and the comparison
depends strongly on the implementation (Section~\ref{sec:compiled}).  With the
compiled ADRT of \cite{Otness23} on one core, the normal-equation polynomial
is the fastest of the three methods up to $N=128$.  From $N=256$ on the oblique
inverses are faster, the hybrid by $1.8$ and $3.3$ times at $N=256$ and $512$
(Table~\ref{tab:compiled}).

A symbol preconditioner does not change this.  The preconditioned operator is
$E^\top F_\tau E$ in the limit $\tau\to\infty$ (Remark~\ref{rem:filter}), whose
conditioning is $\tau/\lambda_{\min}({\nR}^\top\nR)$ (Section~\ref{sec:cond}), so
the preconditioned polynomial has the cost and structure of $L^n_\tau$.

In summary, the oblique construction contributes a left inverse that is exact
at every degree, with every factor in closed form, a better asymptotic
exponent, and, in compiled single-core implementations, a lower cost than the
normal-equation polynomial from $N=256$ on.
\end{remark}
\section{The conditioning of the frame operator}\label{sec:cond}
The outer polynomial degree $D$ is the one quantity in Theorem~\ref{thm:fast}
for which no a priori bound was given: $\lambda_{\max}(E^\top F_pE)\le4(1+\varepsilon)$
is explicit, so $D$ is governed by $\lambda_{\min}(E^\top F_pE)$.  This section
reports what determines it.  The experiments (Appendix~\ref{app:code}) use the Tikhonov filter, to which the
polynomial filters agree to $\varepsilon/\tau$.

\subsubsection*{Straight lines and the continuum.}
Replace the digital lines by the exact lines $\delta_s(t)=st/(N-1)$ in
\eqref{eq:Ak}, i.e.\ $A_k(s,t)=e^{\mathrm i\omega_kst/(N-1)}$, keeping everything
else (window, quadrants, filters).  For this straight-line transform
$\lambda_{\min}(R^\top R)=1.98N$ for every $N\le128$ and
$\lambda_{\min}(E^\top F_\tau E)=0.911,0.907,0.908$ for $N=8,16,32$ at $\tau=0.2N$, which is
the continuum value: each quadrant is the Fourier multiplier
$\mathbf 1_{\mathcal C_q}(\xi)\,c|\xi|^{-1}/(c|\xi|^{-1}+\tau)$ of its cone, the four cones
cover the plane, and at the Nyquist frequency the eigenvalue of $R^\top R$ is
$2\pi N/\pi=2N$, so $\lambda_{\min}(E^\top F_\tau E)=2N/(2N+0.2N)=0.909$.  For straight lines
$\kappa(E^\top F_\tau E)$ is therefore bounded, $D=\cO(\log\varepsilon^{-1})$, and the
construction of Section~\ref{sec:poly} has cost $\cO(N^2\log^{3/2}N\log\varepsilon^{-1})$.

\subsubsection*{The smallest singular value of the ADRT\@.}
With digital lines $\lambda_{\min}(E^\top F_\tau E)$ decays, also in a periodic
variant without the window $E$, and the minimizer is not a plane wave (the
plane-wave Rayleigh quotients stay above $1.3$).  Asymptotically the decay is
that of the ADRT itself: on the directions where $M_q$ has eigenvalues
$\ll\tau$ the filter is $\lambda/\tau$, so
$\lambda_{\min}(E^\top F_\tau E)\approx\lambda_{\min}(R^{n\top}R^n)/\tau$ once
$\lambda_{\min}(R^{n\top}R^n)\lesssim\tau$.  At $\tau=N$ the minimizers of the two
operators coincide from $N=64$ on (overlap $0.94$, $0.97$ at $N=64,128$) and
$\lambda_{\min}(E^\top F_\tau E)=0.40,0.30$ against $\lambda_{\min}(R^{n\top}R^n)/\tau=0.79,0.53$;
at $\tau=0.2N$ the crossover is near $N=10^3$.  And
\[
\lambda_{\min}(R^{n\top}R^n)=15.5,\ 24.5,\ 35.7,\ 50.3,\ 68.0,\ 88.7\qquad(N=8,16,32,64,128,256),
\]
growing by factors $1.58,1.46,1.41,1.35,1.30$ per doubling, still decreasing,
i.e.\ $\sigma_{\min}(R^n)$ grows like $N^{0.2}$ or slower against $\sqrt{2N}$ for
straight lines.  Hence $\kappa(E^\top F_\tau E)\approx\tau/\lambda_{\min}(R^{n\top}R^n)$
grows like $N^{1/2}$ or slightly faster for $\tau\propto N$, not like $\log N$:
the fit $4/\log_2N$ of Section~\ref{sec:bounds} describes the pre-asymptotic range.

\subsubsection*{Independence of the total degree from $\tau$.}
With $D\approx\sqrt{4\tau/\lambda_{\min}(R^{n\top}R^n)}\,\ln(2/\varepsilon)$ and
$d_k\approx\tfrac12\sqrt{\Lambda_k/\tau}\ln(2/\varepsilon_{\rm in})$,
\begin{align*}
\Big(\sum_kd_k\Big)D\ &\approx\ \sum_k\sqrt{\frac{\Lambda_k}{\lambda_{\min}(R^{n\top}R^n)}}\,\ln\frac{2}{\varepsilon_{\rm in}}\,\ln\frac2\varepsilon\\
&=\cO\Big(\frac{N^{3/2}\sqrt{\log N}}{\sigma_{\min}(R^n)}\log\varepsilon_{\rm in}^{-1}\log\varepsilon^{-1}\Big),
\end{align*}
independently of $\tau$ (observed: $2.0\cdot10^4N$ and $1.8\cdot10^4N$ steps at
$\tau=0.2N$ and $\tau=N$, $N=128$).  The same holds for every filter $r$: since
$\langle F_ru,u\rangle\le\max r\cdot\langle R^{n\top}R^nu,u\rangle$, a frame operator
with $\lambda_{\min}\ge c$ needs $\max r\ge c/\lambda_{\min}(R^{n\top}R^n)$, i.e.\ the
filter must resolve $M_k$ down to eigenvalues of order $\lambda_{\min}(R^{n\top}R^n)$,
at the corresponding polynomial degree.  The cost of the family is therefore
$\cO(N^{5/2}\log^{3/2}N\log\varepsilon^{-1}/\sigma_{\min}(R^n))$ at fixed $\varepsilon_{\rm in}$: $\cO(N^2\log^{3/2}N)$ for
straight lines and, with the measured $\sigma_{\min}(R^n)\sim N^{1/4}$, about
$\cO(N^{2.25}\log^{3/2}N)$ for the ADRT.  In the range $N\le256$ the measured steps
per application ($1.04,1.35,1.60,1.92,2.24,2.72\cdot10^4N$) grow by $1.17$--$1.3$
per doubling, which does not distinguish $N\log N$ from $N^{1.26}$; the mechanism
decides for the power.  Press's power $q=2$ is thus attained by the oblique
inverse for straight lines and missed by $\sigma_{\min}(R^n)^{-1}$ for the ADRT.

\subsubsection*{Minimizing eigenvectors of $R^{n\top}R^n$.}
The minimizers of $R^{n\top}R^n$ (``ghosts'') are symmetric under the dihedral
group, so that all four quadrants see them equally; their Fourier transform is
concentrated on the eight frequencies $(\pm a,\pm b),(\pm b,\pm a)$ with
$2\pi b/N=2.16,2.06,2.11$ for $N=32,64,128$, i.e.\ at $\omega\approx2\pi/3$, and
$83\%$ of their sinogram energy lies within two grid frequencies of $2\pi/3$.
They sit inside the Fourier cone of quadrants \RN{1}/\RN{2}, where a straight
line would give the eigenvalue $2\pi N/\omega\approx190$ ($N=128$), but the digital
quadrant sees them with $17$.  The frequency $2\pi/3$ is where $2\sin\frac\omega4=1$, so
that $z^3=1$ and $A_\omega(s,t)=z^{\delta_s(t)\bmod3}$ takes three values and every
entry of $M_\omega$ has modulus $\ge1$ (Lemma~\ref{lem:kernel}); single quadrants
are not nearly singular there ($\sigma_{\min}(A_\omega)=1.3$--$1.6$), the four cones
fail to cover jointly.

\subsubsection*{Integer sampling.}
The mechanism is the sampling of lines at integer offsets, not the dyadic
structure of the ADRT.  For the Bresenham lines $\delta_s(t)=\lfloor st/(N-1)+\tfrac12\rfloor$ one finds
$\lambda_{\min}(R^\top R)=15.8,28.9,49.9$ for $N=16,32,64$, the same as for the ADRT
($24.5,35.7,50.3$).  Any line sampled at integer offsets carries the rounding
phase $e^{-\mathrm i\omega\{\sigma t\}}$, whose Fourier coefficients
$|c_m(\omega)|=2|\sin\frac\omega2|/|2\pi m-\omega|$ have a divergent sum, so that at
$\omega$ of order one a superposition of aliases $\eta+2\pi m\sigma$ can cancel the
slice of every slope; a straight line with exact phases has no aliases.  The
continuum floor survives only well below the Nyquist frequency: restricted to
images band-limited to $c_f$ times Nyquist in both coordinates,
$\lambda_{\min}(R^{n\top}R^n)/N=6.0,6.3,6.3,6.4$ for $c_f=0.25$ ($N=16,\dots,128$),
$3.1,3.0,2.5,2.0$ for $c_f=0.5$, $3.1,2.1,1.4,1.0$ for $c_f=0.6$, and the unrestricted
values from $c_f=0.75$ on (the ghosts live at $|\xi|\approx N/3$).  A comparison
theorem with the continuous Radon transform, whose conditioning is known
\cite[Thm.~II.5.1]{Natterer}, can therefore certify $\lambda_{\min}\ge cN$ on the
lower third of the spectrum but says nothing about the ghosts.

\subsubsection*{Unsuccessful remedies.}
Three remedies were tested and fail.  (i) A frequency-dependent $\tau_k$ that is
small in a band around $2\pi/3$: at $N=64$, $\tau=N$, $\lambda_{\min}$ rises only from
$0.40$ to $0.48,0.52,0.58$ for bands of $\pm2,4,8$ grid frequencies, because the
weak directions are spread over $\omega\in[1.8,2.3]$.  (ii) Deflation of the lowest
eigenvectors of $E^\top F_\tau E$: the low tail is a bulk, not a few modes ($16$
eigenvalues below $0.9$, $143$ below $1.0$, $590$ below $1.2$ at $N=64$, $\tau=0.2N$).
(iii) A shifted outer solve $(E^\top F_pE+\mu)^{-1}$, which has the explicit
interval $[\mu,4+\mu]$: since the whole spectrum is $\cO(1)$, the bias
$\mu(E^\top F_pE+\mu)^{-1}$ touches every mode, and the noiseless Shepp--Logan
error becomes $5.6\%$ at $\mu=0.1$ and $15\%$ at $\mu=0.3$.  The exact inverse with
the a posteriori Weinstein enclosure remains the only rigorous choice of $D$
that costs no accuracy.

\subsubsection*{Conditioning of the hybrid.}  The smallest eigenvalue of the
hybrid's outer operator $G_1^{-1}G_h$ (Section~\ref{sec:hyb}) decays like
that of $E^\top FE$: $0.263$, $0.208$, $0.164$, $0.127$ ($\tau=0.2N$) and
$0.139$, $0.107$, $0.083$, $0.063$ ($\tau=N$) for $N=8,\dots,64$, factors $0.77$--$0.80$ per
doubling in both constructions.  This is expected: with $w=\Sone x/\|\Sone x\|$,
\[
\langle\widehat U^\top\Phi'\,\widehat Uw,w\rangle\le\max p\cdot\frac{\|\nR x\|^2}{\|\Sone x\|^2} ,
\]
and for the ghosts above (offset frequency near $2\pi/3$, where
the level-1 singular values are between $1$ and $\sqrt3$) the denominator is
$\cO(\|x\|^2)$ while $\|\nR x\|^2=\lambda_{\min}({\nR}^\top\nR)\|x\|^2$.  The conditioning of
the hybrid is therefore again $\approx\tau/\lambda_{\min}({\nR}^\top\nR)$ up to a
constant, and the total degree $(\sum_kd_k)D\sim N^{3/2}/\sigma_{\min}(\nR)$ is
unchanged in order.  The explicit level-1 inverse cannot help with the ghosts
because they are not a level-1 phenomenon.

\subsubsection*{Per-frequency lower bounds.}
The simplest rigorous route would replace each quadrant Gram by its smallest
eigenvalue, $M_\omega\succeq\sigma_{\min}(A_\omega)^2\Id$, so that
${\nR}^\top\nR\succeq W:=2\sum_{\omega}\sigma_{\min}(A_\omega)^2\big(P^v_\omega+P^h_\omega\big)$
with $P^{v,h}_\omega$ the vertical and horizontal offset-frequency projections
on the padded torus, an operator whose smallest eigenvalue is computable.
It is useless: $\sigma_{\min}(A_\omega)^2$ is $2.1$--$2.3$ for $|\omega|\ge2\pi/3$ and
$0.28$--$0.31$ on $[\pi/2,2\pi/3)$, but below $|\omega|=\pi/2$ it collapses
($\le4\cdot10^{-3}$ on $[\pi/4,\pi/2)$, $0$ to working precision below $\pi/4$), so
$\lambda_{\min}(W)\le4\cdot10^{-3}$ for $N\ge8$ while $\lambda_{\min}({\nR}^\top\nR)=15.5,24.5,35.7$.  The images that defeat $W$ are smooth in
both directions, and for those the Gram $M_\omega$ is large on the relevant
vectors (the continuum regime) even though its smallest eigenvalue vanishes.
A bound therefore has to be two-dimensional: it must control $M_\omega$ on the
subspace of column vectors with low horizontal frequency.  The continuum
comparison does this only for $|\xi|\lesssim0.4\pi$ (the band-limited floor
above), while the per-quadrant singular-value floor starts at $|\omega|\ge\pi/2$;
the ghosts, at $|\xi|\approx2.2$, sit in the gap between the two regimes, which
is why neither argument reaches them.

\begin{remark}[Fusing through more levels]\label{rem:m0}
With the filter on levels $m_0+1,\dots,n$ and the fusion through
$\npmz{S}{m_0}\cdots\Sone$ ($\tau=0.2N$): the
inner degree sums drop to $83,53,33$ ($/N$, $N=32$) for $m_0=1,2,3$, while the
unpreconditioned outer conditioning rises to $7.7,18.2,42.8$, so that
$(\sum_kd_k)\sqrt\kappa=230,226,216$ is essentially flat; the preconditioner
$({\npmz{S}{m_0}}^\top\cdots{\Sone}^\top\Sone\cdots\npmz{S}{m_0})^{-1}$ would bring $\kappa$
to $5.4,7.3,6.5$ and the product to $194,144,84$, but for $m_0\ge2$ it is the
Gram of several levels and has no explicit decomposition (for $m_0=n$ it is
$({\nR}^\top\nR)^{-1}$ itself).  The same pattern holds at $N=8,16$.  The level-1
hybrid is thus the member of this family for which the preconditioner is
explicit, and the gain from going further is limited by the lack of one.
\end{remark}

\subsubsection*{Summary.}
The conditioning of the oblique inverses is not a property of the filters or
of the window but of the ADRT: $\kappa(E^\top F_\tau E)\approx\tau/\lambda_{\min}(R^{n\top}R^n)$,
and $\sigma_{\min}(R^n)$ is set by the aliasing of integer-sampled lines near
$\omega=2\pi/3$, a mechanism shared by every integer-line Radon transform and
absent for exact lines.  A rigorous a priori bound on $D$ that does not inflate
the cost is equivalent to a lower bound on $\sigma_{\min}$ of the ADRT of the
right order; the continuum comparison gives it on the lower third of the
spectrum, and the pole bound $\sigma_{\min}(A_\omega)\ge(2\sin\frac\omega4)^n\ge1$ for
$|\omega|\ge2\pi/3$ gives an $\cO(1)$ floor per quadrant at the ghost frequencies,
which is too weak by $\sigma_{\min}(R^n)^2$.  Closing this gap is the open problem.
\section{Discussion}\label{sec:disc}
The construction answers, for oblique projections, the question of an
explicit inverse of the ADRT for out-of-range data: $L^n_\theta$, $L^n_\tau$ and
$L^n_p$ are exact left inverses written in closed form from the butterfly
representation of the single-quadrant levels, with a single well-conditioned
SPD operator inverted by a fixed polynomial, whose degree is governed by
$\sigma_{\min}$ of the ADRT (Section~\ref{sec:cond}).  What remains non-explicit is the
orthogonal projection: $(\nR)^\dagger$ requires $({\nR}^\top\nR)^{-1}$, and
${\nR}^\top\nR=\sum_dZ^{(d)}\otimes K_d+K_d\otimes Z^{(d)}$ (with $Z^{(d)}$ the shift by $d$
and explicit integer generators $K_d(j,j')=\#\{s:\delta_s(j')-\delta_s(j)=\pm d\}$)
is a two-level Toeplitz-like operator for which no superfast direct solver
is known \cite{ChanJin}; Press's multigrid \cite{Press06} and the iterations
compared in \cite{LORR} remain the methods of choice for it.  The family
$L^n_\tau$ shows that the gap is one of degree: as $\tau\to\infty$ the oblique
inverse tends to $(\nR)^\dagger$ while $\kappa(E^\top F_\tau E)$ tends to
$\kappa({\nR}^\top\nR)$; at $\tau=N$ one is within $7\%$ of Moore--Penrose at
$\kappa\approx13$.  Remark~\ref{rem:normal} puts the construction in perspective: a
fixed-degree Chebyshev polynomial of the normal equations, of degree
$\Theta(N/\sigma_{\min}(\nR))$, reaches the Moore--Penrose solution and, with a
compiled ADRT, is the fastest method up to $N=128$; the oblique inverses are
asymptotically cheaper by $\sqrt N/\sqrt{\log N}$, faster from $N=256$ on (the
hybrid by $3.3$ times at $N=512$), and exact at every degree, and both carry
the same factor $\sigma_{\min}(\nR)^{-1}$.

The hybrid of Section~\ref{sec:hyb}, SPIFE on the cross-quadrant level and
the polynomial filters on the upper levels, is an exact left inverse that is
better than $L^n_p$ by constant factors: $24$--$30\%$ less work per application at equal
$\tau$ and $\varepsilon$ in the fast implementation (Section~\ref{sec:fast}), and
noise amplification $1.14$--$1.21\times$ Moore--Penrose instead of
$1.30$--$1.38\times$ at $\tau=0.2N$.  The two ingredients that make it
work are the exact fusion of the quadrants at level 1 through ${\Sone}^\top$ and
the explicit $G_1^{-1}$ (Lemma~\ref{lem:g1}) as a preconditioner; the
truncated level-1 inverse of \cite{LORR} (Lemma~\ref{lem:lorr}) is cheaper still but
preconditions less well.  It does not change what
governs the conditioning: the outer spectrum still shrinks like
$\lambda_{\min}({\nR}^\top\nR)/\tau$, because the weak directions of the ADRT live
in the upper levels.  A natural extension, fusing through the first $m_0$
levels with the filter on levels $m_0+1,\dots,n$, lacks an explicit
preconditioner for $m_0\ge2$ (the Gram of two levels is no longer explicit);
Remark~\ref{rem:m0} reports its unpreconditioned conditioning.  The growth
$\sigma_{\min}(R^n)\sim N^{1/4}$ is measured but not proved, and proving it
remains open.  The asymptotic exponent of both constructions depends on it.

\section*{Acknowledgements}
The author used Claude (Anthropic), with the models Claude Fable~5.1 and Claude
Opus~5.5, as an AI assistant for code development, numerical experiments, and
drafting and editing of the manuscript.  All results were checked by the
author, who takes full responsibility for the content.

\appendix
\section{List of notation}\label{app:notation}
Symbols are listed in order of appearance within each group, with the place
where they are defined.

{\small
\begin{longtable}{@{}p{0.25\textwidth}p{0.57\textwidth}p{0.12\textwidth}@{}}
\toprule
symbol & meaning & defined in\\
\midrule
\endfirsthead
\toprule
symbol & meaning & defined in\\
\midrule
\endhead
\bottomrule
\endfoot
\multicolumn{3}{@{}l}{\emph{General}}\\[2pt]
$N=2^n$ & image size; $\log$ is the logarithm of base two & \S\ref{sec:prelim}\\
$[m]$, $[m]_1$ & $\{0,\dots,m-1\}$ and $\{1,\dots,m\}$ & \S\ref{sec:prelim}\\
$a\brem b$ & remainder of $a$ modulo $b$ & \S\ref{sec:prelim}\\
$\Id_m$, $\Id$ & identity matrix & \S\ref{sec:prelim}\\
$x(i,j)$ & image entries, $i$ the row index & \S\ref{sec:prelim}\\
$A^\dagger$, $A^+$ & Moore--Penrose pseudo-inverse & \S\ref{sec:prelim}\\
$\kappa(\cdot)$, $\sigma_{\min}(\cdot)$ & condition number, smallest singular value & \\
$\rg$, $\ker$ & range, kernel & \\
$\nu$ & data noise, $b=\nR[x]+\nu$ & \S\ref{sec:left}\\[4pt]
\multicolumn{3}{@{}l}{\emph{The ADRT}}\\[2pt]
$D_n(h,s)$ & digital line of offset $h$ and slope $s$ & \eqref{eq:dline}\\
$\delta_s(t)$, $w_j(s)$ & row of the digital line in column $t$; its bit weights & \eqref{eq:delta}\\
$E^n_m$, $E^n$ & index sets of level-$m$ data and of all lines & \eqref{eq:Rnn}\\
$R^n_n$ & single-quadrant ADRT & \eqref{eq:Rnn}\\
$\nm{S}$ & single-quadrant level $m$ & \eqref{eq:Snm}\\
$T^n_q$, $q=\RN{1},\dots,\RN{4}$ & quadrant permutations & \eqref{eq:permute}\\
$\nR$ & full ADRT, $\npmz{S}{n}\cdots\npmz{S}{1}$ & \S\ref{sec:prelim}\\
$\Sone$, $\npmz{S}{m}$ & cross-quadrant level; single-quadrant levels $\Id_4\otimes\nm{S}$ & \S\ref{sec:prelim}\\
$(\nR)^\dagger_{\rm SPIFE}$ & explicit inverse of \cite{LORR} for data in the range & \eqref{eq:spife}\\[4pt]
\multicolumn{3}{@{}l}{\emph{Offset-Fourier domain}}\\[2pt]
$L$ & cyclic period of the offsets, $L=2N$ & \S\ref{sec:fourier}\\
$\cR^n_n$ & cyclic single-quadrant ADRT & \eqref{eq:cyc}\\
$\iota$, $E$ & embedding of data into cyclic strips; zero padding of images & \S\ref{sec:fourier}\\
$\mathrm{DFT}_h$, $\omega_k$, $z$ & DFT in the offset; frequency $2\pi k/L$; $z=e^{\mathrm i\omega_k}$ & \S\ref{sec:fourier}\\
$A_k$ ($A_\omega$) & single-quadrant transform at offset frequency $\omega_k$ & \eqref{eq:Ak}\\
$\mathcal B^n_m$, $\beta$ & butterfly of level $m$ and its $2\times2$ block & \eqref{eq:butterfly}\\
$M_k$ ($M_\omega$) & Gram matrix $A_k^*A_k$ & \S\ref{sec:fourier}\\
$\mathcal C_q$ & Fourier cone of quadrant $q$ & \S\ref{sec:fourier}\\[4pt]
\multicolumn{3}{@{}l}{\emph{Oblique inverses}}\\[2pt]
$\rho(L)$ & stability ratio $\|L\|_2\,\sigma_{\min}(\nR)$ of a left inverse & \S\ref{sec:left}\\
$L_W$ & quadrant combination of exact single-quadrant inverses & \eqref{eq:LW}\\
$\theta$, $P_k$ & threshold; spectral projector of $M_k$ & \S\ref{sec:partition}\\
$\Phi_q$, $F$ & quadrant projectors and frame operator & \eqref{eq:Phi}\\
$\cR_q$, $M_q$, $r(M_q)$ & cyclic transform and Gram operator of quadrant $q$; function of it & \eqref{eq:Mq}\\
$\cR_q^+$ & quadrant-wise cyclic pseudo-inverse & \S\ref{sec:partition}\\
$L^n_\theta$ & spectral-partition oblique inverse & \eqref{eq:Ltheta}\\
$\tau$, $F_\tau$, $L^n_\tau$ & Tikhonov parameter, frame operator and inverse & \eqref{eq:Ltau}\\[4pt]
\multicolumn{3}{@{}l}{\emph{Polynomial filters}}\\[2pt]
$p_k$, $d_k$ & Chebyshev filter at frequency $\omega_k$ and its degree & \S\ref{sec:poly}\\
$\Lambda_k$ & upper bound for $\lambda_{\max}(M_k)$ & \eqref{eq:Lambda}\\
$\Lambda^{\rm cert}_k$, $r_m(\alpha)$ & certified bound; row-sum recursion & \eqref{eq:cert}\\
$F_p$, $L^n_p$ & polynomial-filter frame operator and inverse & \eqref{eq:Lp}\\
$\varepsilon$, $\varepsilon_{\rm in}$ & outer and inner polynomial tolerances & Thm.~\ref{thm:fast}\\
$D$, $c_D$ & degree of the outer Chebyshev polynomial; the polynomial & \S\ref{sec:stab}\\[4pt]
\multicolumn{3}{@{}l}{\emph{Hybrid inverse}}\\[2pt]
$Q$ & upper levels $S^n_n\cdots S^n_2$ & \S\ref{sec:hyb}\\
$A'_k$, $M'_k$ & upper-level butterflies and their Gram matrix & \S\ref{sec:hyb}\\
$\Phi'$, $G_1$, $G_h$ & filtered upper-level Gram; level-1 Gram; hybrid Gram & \eqref{eq:hybnot}\\
$\widetilde G$, $L_{\rm hyb}$ & preconditioned outer operator; hybrid inverse & \eqref{eq:hyb}\\
$K$, $J$ & $\mathrm{tridiag}(1,2,1)$; pair swap & Lemma~\ref{lem:g1}\\
$\Pi$, $M$ & $2\times2$-cell average; truncation of level-1 data & Lemma~\ref{lem:lorr}\\
\end{longtable}
}

\section{No trigonometric diagonalization of the normal operator}\label{app:galois}
The levels of the ADRT have explicit singular value decompositions whose
singular values, $2\cos\frac{k\pi}{4N+2}$ \cite[Thm.~3.5]{LORR} or, in the
offset-Fourier domain, $2|\cos\frac\omega4|$ and $2|\sin\frac\omega4|$, are values
of trigonometric functions at rational multiples of $\pi$ and hence lie in
cyclotomic fields.  The composition of the levels does not have this property.

\begin{proposition}[exact computation]\label{prop:galois}
Let $G={\nR}^\top\nR$, an integer matrix.  For $N=4$ its characteristic
polynomial factors over $\QQ$ as
\[
\begin{aligned}
&(\lambda-9)(\lambda-11)^2(\lambda-13)\,(\lambda^3-39\lambda^2+494\lambda-2036)\\
&\qquad\cdot(\lambda^3-45\lambda^2+630\lambda-2780)^2\,(\lambda^3-83\lambda^2+1750\lambda-10948),
\end{aligned}
\]
where the three cubics are irreducible with discriminants $7060$, $191700$ and
$7727188$, none of which is a square, so that each has Galois group $S_3$.  For
$N=8$ the irreducible factors have degrees $6,6,10,10$ and $16$, the last with
multiplicity two, and both sextics have Galois group $S_6$.
\end{proposition}
The factorization, the discriminants and the Galois groups are computed in
exact integer arithmetic (characteristic polynomial and factorization over
$\QQ$, Galois groups of the factors of degree at most six), in a fraction of a
second for both sizes.

\begin{corollary}\label{cor:galois}
Let $\mathbb K$ be an abelian extension of $\QQ$, for instance a cyclotomic field
$\QQ(e^{2\pi\mathrm i/m})$, and let $Q$ be an invertible matrix with entries in
$\mathbb K$.  Then $Q^{-1}GQ$ is not diagonal for $N=4$.  In particular no product
of discrete Fourier, cosine, sine, Hartley or Walsh--Hadamard transforms (whose
entries, normalizations included, lie in cyclotomic fields) diagonalizes $G$,
the singular values of $\nR$ are not expressible by trigonometric functions of
rational multiples of $\pi$, and for $N=8$ some of them are not expressible by
radicals at all.
\end{corollary}
\begin{proof}
If $Q^{-1}GQ=D$ is diagonal, then $D$ has entries in $\mathbb K$, so every eigenvalue
of $G$ lies in $\mathbb K$.  Let $\lambda$ be a root of one of the cubics of
Proposition~\ref{prop:galois}.  Every subfield of an abelian extension of $\QQ$
is a Galois extension of $\QQ$ with abelian group, so $\QQ(\lambda)/\QQ$ is Galois;
then the cubic, being irreducible with a root in $\QQ(\lambda)$, splits there, its
splitting field is $\QQ(\lambda)$, and its Galois group is the abelian group of
$\QQ(\lambda)/\QQ$, of order $3$.  This contradicts the Galois group $S_3$.
Singular values of $\nR$ are square roots of eigenvalues of $G$, so the same
argument excludes them from $\mathbb K$.  Finally $S_6$ is not solvable, so the roots
of the sextics for $N=8$ are not expressible by radicals.
\end{proof}

The obstruction therefore arises only from the composition of the levels,
each of which is diagonalized explicitly: the orthogonal projection onto the
range of the ADRT, and with it the Moore--Penrose inverse, cannot be obtained by
composing trigonometric transforms with explicit diagonal scalings, which is
why the inverses of this paper are oblique.

\section{Code}\label{app:code}
The code is available at \url{https://github.com/dsrim/adrt-pinv} and archived on
Zenodo \cite{Rim26code}.  The
experiments share the following modules: \texttt{adrtcore.py} (the ADRT and its
transpose), \texttt{pinv.py} (CGLS), \texttt{specpart.py} (dense spectral
partition), \texttt{fastspec.py} and \texttt{fasthybrid.py} (NumPy versions of
$L^n_p$ and of the hybrid), \texttt{fastops.py} (optimized engine),
\texttt{fastcpp.py} and \texttt{oblique.cpp} (C\texttt{++} row kernels used for all
wall times), and \texttt{lab\_frame.py}, \texttt{lab\_hybrid.py} (dense
laboratories).  The file \texttt{RESULTS.md} in the repository lists the script
behind each table, figure and numerical statement of this paper.

\end{document}